\documentclass[11pt]{article}
\usepackage{amsmath,amssymb,amsthm}
\usepackage{chemformula} 
\usepackage{blkarray}
\usepackage{mathtools}
\usepackage{cases}
\usepackage{booktabs}
\usepackage{subcaption}
\usepackage{authblk}
\usepackage{geometry} 
\usepackage{hyperref}
\usepackage{parskip}
\usepackage{comment}
\usepackage[toc]{appendix}

\usepackage{amsfonts}
\usepackage{xcolor}

\usepackage{chemformula}
 
\usepackage{lineno}

\usepackage{enumerate}

\newtheorem{theorem}{Theorem}[section]
\newtheorem{proposition}[theorem]{Proposition}

 \newtheorem{corollary}[theorem]{Corollary}
\theoremstyle{remark}
\newtheorem{remark}{Remark}
\newtheorem{definition}[theorem]{Definition}

\usepackage{comment}
 
 \usepackage{subcaption}

\title{A numerical study of a PDE-ODE system with \\ a stochastic dynamical boundary condition:\\ a nonlinear model for sulphation phenomena. }

\author{Francesca  Arceci, Daniela Morale, Stefania Ugolini } 
 \affil{ Department of Mathematics, University of Milano, Italy}

\begin{document}
\maketitle

\begin{abstract}
We investigate   the qualitative behaviour of the solutions of a stochastic boundary value problem on the half-line for a nonlinear system of parabolic reaction-diffusion equations, from a numerical point of view. The model describes the chemical aggression of calcium carbonate stones under the attack of sulphur dioxide. The dynamical boundary condition is given by a Pearson diffusion, which is original in the context of the degradation of cultural heritage. We first discuss a scheme based on the Lamperti transformation for the stochastic differential equation to preserve the boundary and a splitting strategy for the partial differential equation based on recent theoretical results. Positiveness, boundedness, and stability are stated. The impact of boundary noise on the solution and its qualitative behaviour both in the slow and fast regimes is discussed in several numerical experiments.    
\end{abstract}

\section{Introduction}

This paper concerns the qualitative and numerical study of a nonlinear reaction diffusion partial differential equation (PDE) system on the half-line coupled with a stochastic dynamical boundary condition. The deterministic counterpart arises in the description of the evolution of the chemical reaction of sulphur dioxide with the surface of calcium carbonate stones in the phenomenon of cultural heritage degradation, \cite{2007_AFNT_TPM,2004_ADN,2008_Bonazza_Sabbioni,2019_BCFGN_CPAA,2024_Comite_morale_ugolini}. Recently, a bit of work has been done for the description of possible sources of randomness in the sulphation model. In particular, at the microscale  an interacting particle model has been proposed,\cite{2024_MACH2023_particles}  and a probabilistic interpretation  of a deterministic model introduced in   \cite{2004_ADN} has been derived,\cite{2024_DMLTSU_Arxiv}.   Here, we discuss a model at the macroscale, first    introduced in   \cite{2024_MACH2023_PDE,2023_Arceci_Giordano_Maurelli_Morale_Ugolini,2023_MMU}. The authors propose a system of partial differential equations coupled with a source of randomness, given by a stochastic dynamical boundary condition. In particular, at the left boundary  a Pearson process is considered, that is a solution to a Brownian motion-driven stochastic differential equation (SDE) with a mean reverting drift and a squared diffusion coefficient, which is a second-order polynomial of the state.

 \medskip

The problem of the degradation of sandstones, limestones, and marble stones with different porosity used as building materials for thousands of years is a very important issue that has been observed in the last century. The first cause can be attributed to atmospheric pollutants, in particular the reaction of sulphur dioxide with calcareous surfaces, which then forms gypsum and black crusts, \cite{2008_Bonazza_Sabbioni,2020_Comite_fermo1,2024_Comite_morale_ugolini}. In the last couple of decades, some mathematical models have been introduced for studying the evolution of the above degradation phenomena, \cite{2018_saba}. 

\smallskip
 
The dynamical evolution of the chemical reaction of sulphur dioxide with the surface of calcium carbonate stones in the cultural heritage degradation process is a complex system, due to the involvement of physical, chemical and biological processes, coupled with specific characteristics of the materials such as capillarity and porosity, \cite{2007_AFNT_TPM,2004_ADN,2019_BCFGN_CPAA}.  More precisely, in    \cite{2004_ADN} it has been introduced a first deterministic model for a quantitative description of the penetration of sulphur dioxide \ch{SO2} into porous materials, as calcium carbonate stones \ch{CaCO3}. It assumes that polluted air and, in particular sulphur dioxide,  diffuses through the pores of the stones and interacts with their surface,  causing  the following reaction \[
 \ch{CaCO3} + \ch{H2SO4}+\ch{H2O}\to\ch{CaSO4}\cdot2\ch{H2O}+\ch{CO2}.
 \]
Given the sulphur dioxide concentration $\rho$ and  the calcite density $c$, the model reads as following, \cite{2007_AFNT_TPM,2004_ADN}
\begin{equation} \label{eq:modello_Natalini} 
	\begin{split}
		\frac{\partial }{\partial t}  \rho   &=	\nabla  \cdot(\varphi  \nabla s )  -\lambda \rho c ,   \qquad   \mbox{ in } (0,\infty)\times (0,T),\\
		\partial_t c &= -\lambda \varphi s c,  
	\end{split}
\end{equation}
with constant initial conditions
\begin{equation} \label{eq:modello_Natalini_condition1} 
	\begin{split}
  \rho(x,0)&= \rho_0, \quad
c(x,0) = c_0;
	\end{split}
\end{equation}
and  constant boundary condition $
\rho(0,t)=\bar{\rho}_0.$
Function  $s$ stands for the porous concentration of sulphur dioxide and is such that $\rho = \varphi s$, whereas $\lambda\in \mathbb R_+$ is the reaction rate.  The function $\varphi=\varphi(c)$ is the porosity of the material, which is usually a function of calcium carbonate. The interested reader may refer to    \cite{2004_ADN}  and to the references therein,  for a detailed derivation of the PDE system, and a qualitative behaviour of the solution on the half-line.   
 
 \smallskip
 
The model \eqref{eq:modello_Natalini} has been the starting point of a widespread literature, both analytical and numerical,  \cite{2007_AFNT_TPM,2007_AFNT2_TPM,2019_BCFGN_CPAA,2023_BCFGN_NLAA,2005_GN_NLA,2007_GN_CPAA}. To our knowledge, the study of degradation phenomena has been utilised primarily for  pure statistical models or deterministic partial differential equations, \cite{2018_saba}. Usually, the system has been studied on the half-line $\mathbb R_+$ and the boundary conditions are either constant or a given pre-selected bounded measurable positive function, \cite{2004_ADN,2005_GN_NLA}.

 \smallskip
 
Our contribution within this research area has been the introduction of a random dynamical boundary condition  $\rho(t,0)=  \Psi_t,$ where $\{\Psi_t\}_{t\in [0,T ]}$ is the solution of the following stochastic differential equation (SDE) of the It\^o type 
\begin{equation}\label{stochastic_boundary_condition_intro}
 d\Psi_t= \alpha(\gamma-\Psi_t) dt +\sigma \sqrt{\Psi_t\left(\eta-\Psi_t\right)}dW_t,    
\end{equation}
with $\alpha,\sigma,\gamma,\eta\in \mathbb R_+$, and $\gamma\le \eta$, \cite{2023_Arceci_Giordano_Maurelli_Morale_Ugolini,2024_MACH2023_PDE,2023_MMU}.  The choice of such an equation is due to a preliminary study of the time series of \ch{SO2} concentration in the Milano area, \cite{2023_Arceci_Giordano_Maurelli_Morale_Ugolini}, carried out with the same statistical techniques as in   \cite{2021_Giordano-Morale}. This study shows clear evidence of random fluctuations, boundedness, and a mean reversal effect. The  way we have chosen to introduce a bounded noise, even though  driven by an unbounded  Wiener process $W$ is to consider a case of  Pearson process,  solution of the  equation \eqref{stochastic_boundary_condition_intro}. The issue of considering a  bounded noise, is very important above all in applications, \cite{2013_donofrio,2017_donofrio_flandoli}. Pearson diffusions are often considered in financial applications; for instance,   the very special case of Cox–Ingersoll–Ross process is well known; furthermore, they are often involved in the evolution of interest rates, \cite{2002_delbaen_interest_rate} or as a model for correlation, \cite{2016_teng}; recently,  an example also in the context of the first passage time is provided, \cite{2023_sacerdote}. We also mention that Pearson processes have been popular in applications such as population genetics,
under the name of Wright-Fisher diffusion, \cite{2010_Griffiths,2018_Griffiths,2013_Pal}.
 
 \smallskip
Hence, the novelty we propose here is the description of the evolution of air pollutants, such as sulphur dioxide via a version of a Pearson process as the dynamical boundary condition of a partial differential system which is nonlinear and not strongly parabolic, describing the evolution inside the material. Evolution problems with stochastic dynamical boundary conditions have been studied with particular attention to the case in which at the boundary the solution acts as a white noise, \cite{2015_barbu_bonaccorsi_Tubaro_JMAA,2006_Bon_Zig,2015_Brzezniak_russo,2004_chueshov,1993_daPrato_Zabczyk,2007_debussche_Fuhrman_Tessitore}. To the best of our knowledge, there are a couple of studies closer related to our setting in   \cite{2006_Bon_Zig} and \cite{2004_chueshov}, where the solution of the PDE satisfies a SDE at the boundary. We stress that the analysed experimental data for \ch{SO2} are not compatible with a white noise behaviour. Furthermore, since we do not use Neumann type boundary conditions which could admit a white noise type condition, a process driven by a Brownian motion seems more appropriate in the case of a Dirichlet boundary condition. From the mathematical point of view, the regularity of the noise considered here is better than the regularity of the white noise, even though it is much more irregular than the boundary conditions assumed in the existing literature,\cite{2019_BCFGN_CPAA,2005_GN_NLA,2007_GN_CPDE,2023_BCFGN_NLAA}. Finally, even though the noise is not Lipschitz, it is naturally bounded, which is important from a modelling point of view. 
 \smallskip
 
In the present paper we perform a complete qualitative and numerical analysis of the solution behaviour of the system \eqref{eq:modello_Natalini} coupled with the stochastic boundary condition, given by the stochastic process $\Psi$ in \eqref{stochastic_boundary_condition_intro}. We aim to illustrate the role of randomness in the dynamical evolution of the cited chemical reaction.  
From an analytical point of view, we allowed the boundary condition to have lower regularity, differently from the case in \cite{2007_GN_CPDE}; indeed, it inherits the same regularity of the Wiener process $W$. Thus, the  process $\Psi$ has almost surely trajectories in $C^\beta$, for every $\beta\in(0,1/2)$, i.e. the class of  H\"older continuous functions of order less than one half \cite{2023_MMU}. Furthermore, the SDE has not Lipschitz diffusion coefficient.  The well-posedness, the global existence and the pathwise uniqueness of a mild solution of system  \eqref{eq:modello_Natalini}, endowed with the stochastic boundary condition \eqref{stochastic_boundary_condition_intro}, have already been established in   \cite{2023_MMU}.

\smallskip

For the convenience of the reader,  we propose a self-contained presentation of the relevant mathematical properties associated with the peculiar stochastic boundary condition, in particular its boundedness. Furthermore, we face two problems arising in the application of the Eulero-Maruyama scheme for its numerical simulation: the not Lipschitz diffusion coefficient and the fact that the scheme is not domain preserving; to overcome such problems, we consider a particular spatial change transform of the process which allows to get rid of the not Lipschitz diffusion coefficient. The resulting additive noise It\^o  equation shows a constant diffusion coefficient but a not Lipschitz drift, which is also periodic and monotone decreasing. We recall that super growing coefficients in the Eulero-Maruyama scheme produce divergent solutions, \cite{2021_chen_Gan_wang}. Furthermore, in the case of the one-sided Lipschitz continuous drift, with the derivative
growing at most polynomially and global Lipschitz diffusion coefficient, one might consider a \emph{tamed}  Euler method, where the drift coefficient is modified to obtain a uniformly bounded coefficient \cite{2012_kloeden}. Recently,  
 a smooth truncation procedure has been proposed to overcome both the difficulty caused by the not Lipschitz drift coefficient, which has less regular properties, \cite{2023_chen_LSST}, and to guarantee that monotonicity is maintained. We discuss the qualitative behaviour of the solution of the SDE and we provide a related error analysis. 

\smallskip

Finally, we describe the original part of the paper. The numerical solution of system \eqref{eq:modello_Natalini} with the stochastic dynamical boundary condition given by the process
\begin{equation}\label{eq:modello_Natalini_stochastic_boundary}
\rho(0,t)=\Psi_t
\end{equation}
is investigated, providing a fully discrete scheme based on the splitting strategy proposed in   \cite{2023_MMU}.  Indeed, in order to study the well-posedness of the process $(s,c)$, where $s=\rho/\varphi(c)$ is the solution of the equation
 \begin{align*} 
\partial_t s = \partial_x^2 s +\frac{\varphi_2}{\varphi(c)} \partial_x c\partial_x s -\lambda c \left( \varphi_2 s+1\right) s
\end{align*} 
with initial condition $s(x,0)=\rho_0/\varphi(c_0)$ and stochastic boundary condition $s(t,0)=\widetilde{\Psi}=\Psi_t/\varphi(c(t,0)),$ the authors of the cited reference choose to decompose the variable $s$ as the sum of two random functions
$u$ and $v$.
The random function $u$ is taken as the solution of a diffusion heat equation, which inherits the stochastic dynamical boundary condition $\widetilde{\Psi}$ and with zero initial condition. Consequently, the random function $v$ remains the solution of a nonlinear and nonlocal partial differential equation endowed with zero boundary condition and  initial condition $s(x,0)$. The advantage of the strategy is that one imposes the irregular boundary condition to the heat equation, which is mathematically
very well understood. The equation for $v$ is highly nonlinear, coupled with $u$, but it has  deterministic initial   and   boundary condition. Due to the effect of  diffusion, the highly oscillating boundary condition becomes smoother in the indoor environment and, as a final result, the smoothing power of the Laplacian becomes helpful for the whole system, also including the reaction part, \cite{2024_MACH2023_PDE,2023_MMU}.

\medskip

The discretization scheme takes advantage of the same splitting strategy. First, a discrete heat equation via the forward time centred space (FTCS)  approximation of the solution for $u$ is considered. It is coupled with a zero initial condition and  a discrete stochastic path as boundary condition. The latter is obtained via a Lamperti transformation, to overcome the lack of monotonicity and the problem that the Euler-Maruyama approach does not preserve the bounded domain, \cite{2023_chen_LSST,1992_Kloeden_Platen,2012_kloeden}. Given the discrete random diffusion, an FTCS approximation is considered for the nonlinear bivariate process $(v,c)$.  A preliminary analysis is performed to preserve the positivity and boundedness of the solution. An upper bound for the initial data for $c$ and for the temporal mesh step is provided in dependence on both the spatial space step and the parameters governing the dynamics. Then, a stability analysis is proposed. In the framework of a semi-discretization approach to the same random system, in   \cite{2024_Arceci_DeVecchi_Morale_Ugolini} the theoretical convergence of the discrete scheme here introduced toward the solution to the original continuous one is investigated.\\
A qualitative analysis of the numerical sampling is presented by considering different  pathwise examples in dependence on the characteristic parameters of the model. Furthermore, a statistical estimation of the time-space distribution and its first two moments is provided. The last two analyses are provided both in the cases of slow and fast regimes, by considering increasing values of the parameter $\lambda $  related to the activation energy of the reaction. The simulations permit to appreciate the formation of the moving front.
 
  \smallskip

The paper is organizsed as follows. In Section \ref{sec:stochastic_boundary} we present a complete study of the stochastic dynamical boundary condition, both from the analytical and the numerical point of view. In Section \ref{sec:sulfation_stochastic_boundary} we formulate the problem, recalling the main theoretical properties, and we introduce the splitting strategy, used for the determination of the discrete schemes, in Section \ref{sec:discrete_approximantion}. The main properties of the discrete solution, as well as the stability results, are also discussed here.  Sections \ref{se:numerical_sampling} and \ref{se:numerical_sampling_estimation} are devoted to  sample analysis and  illustration of the role of randomness in the behaviour of the numerical solution.  
 
 \smallskip

\section{A stochastic dynamical boundary condition}\label{sec:stochastic_boundary}

 All processes mentioned throughout the paper are adapted to a filtered probability space $\overline{\Omega}_P:=(\Omega_P, \mathcal{F},\{\mathcal{F}_t\}_t, P )$.
We introduce the stochastic process $\Psi=\left(\Psi_{t}\right)_{t\in[0,T]}, T\in \mathbb R_+$,  which plays the role of the boundary condition for the PDE \eqref{eq:modello_Natalini}.
$\Psi$ is the solution of the following autonomous SDE for $t\in[0,T]$
\begin{equation}\label{eq:SDE_for_SO2}
	d\Psi_{t}   =\alpha(\gamma - \Psi_{t}) dt + \sigma\sqrt{\Psi_{t}(\eta-\Psi_{t})} \, dW_t,
\end{equation} 
with $\alpha,\gamma, \sigma, \eta \in \mathbb R_+$ and $\gamma <\eta$.   $W=\left(W_t\right)_{t\in[0,T]}$ is an adapted  Wiener process, so that the process $\Psi=\{\Psi_t\}_{t\in \mathbb R}$ is a diffusion; in particular, it is a mean reverting process, since the drift has the form
$a(x)=\alpha\left(\gamma -   x\right),$ for $x\in \mathbb{R}$. The noise is characterized by a bounded squared diffusion coefficient which is a second order polynomial of the state, i.e.
\begin{equation}\label{eq:diffusion_Pearson}
	b(x)= \sigma\sqrt{x(\eta-x)}.
\end{equation}  We remark that  $\Psi$ belongs to the class of Pearson processes, known in mathematical biology with  $\eta=1$ as Wright-Fisher diffusion,\cite{2013_Pal}. In our case, we let $\eta \in \mathbb R_+.$ The mean reverting drift is such that on average the solution is pulled to a mean level $\gamma$  at a rate   $\alpha$. Process
fluctuations around the mean level are bounded due to physical constraint; all the parameters in the stochastic model are strictly positive quantities.
\subsection{Existence, uniqueness and boundedness}
\begin{proposition} 
Let $(\Omega_P, \mathcal{F},\{\mathcal{F}_t\}_t, P )$ be a filtered probability space and let $W=\{W_t\}_t$ be a $\mathcal{F}_t$ - Wiener process. Let us assume that  $\alpha,\gamma, \sigma, \eta \in \mathbb R_+$ and $\gamma <\eta$.   Then, for any $x\in [0,\eta]$,  there exists a unique pathwise solution and hence a unique strong solution to the equation \eqref{eq:SDE_for_SO2}.  
\end{proposition}
\begin{proof}
A classical result by Skorokhod \cite{1995_skorokhod} shows the existence of solutions under the condition that the coefficients are only continuous.  Clearly, continuity holds for the drift coefficient $a(x)$, for $x\in \mathbb{R}$, while for the diffusion coefficient $b(x)$, continuity holds for $x\in [0,\eta]$. Furthermore, with a little algebra, it is easy to prove that both  coefficients have a sub-quadratic growth, i.e. there exists a constant $k=k(\alpha,\gamma,\sigma,\eta)$ such that 
\begin{equation*}
    |a(x)|^2 + 
|b(x)|^2  \leq k(1+|x|^2).
\end{equation*}
Therefore, there exists a bounded solution for \eqref{eq:SDE_for_SO2} on $[0,\eta] $ with probability one. Although the drift coefficient is Lipschitz continuous, the diffusion coefficient does not have this desired property; hence, to solve the problem of uniqueness, we apply Theorems IV-1.1, IV-3.2 and its corollary in   \cite{1981_Ikeda_Watanabe}: for the existence of a unique strong solution, it is sufficient to prove the Lipschitz continuity of the drift and the H\"older continuity of the diffusion coefficient, respectively. Indeed,  
\begin{equation*}
    |b(x) - b(y)|   \leq \sigma \sqrt{|\eta x - x^2 - \eta y + y^2|} \leq \sigma \sqrt{|x-y|}\sqrt{|\eta-x-y|} \leq \beta\cdot |x-y|^{1/2},
\end{equation*}
with $\beta=\sigma\eta$. Hence, the diffusion coefficient $b(x)$ is H\"older continuous with index $1/2$. As a consequence, the well-posedness of the SDE is established.
\end{proof}
 
\smallskip

Hypothesis $x\in [0,\eta]$ is not restrictive, since by employing specific constraints on the parameters of equation \eqref{eq:SDE_for_SO2}, one can prove that $x=0$ and $x=\eta$ are entrance boundaries and the solution $\Psi_t \in(0,\eta)$ for $t\in(0,T]$, whenever $\Psi_0 \in[0,\eta]$. 
Let us recall that the diffusion $\Psi=\{\Psi_t\}_{t\in[0,T]}$ is characterized by the following infinitesimal generator on $f\in C_b^2((0,\eta))$
$$
\mathcal{A}  f (x) = a(x)f^\prime(x)+\frac12 b^2(x) f^{\prime\prime}(x).
$$
and by two basic characteristics: the \emph{speed measure} $ m(dx)$ and the \emph{scale function}
$s(x)$.
\begin{definition}
Given $b^2(x)>0$, for $0<x<\eta$, we define the \emph{scale function} $S$ and the \emph{speed measure} $ M(dx)$  with \emph{speed density} $m$, in  $0<x<\eta$ and with $x_0\in (0,\eta)$, by
	\begin{equation*}%\label{eq:speed_scale}
		\begin{split}
		S(x)&= \int^x_{x_0} \exp\left( -\int^y_{x_0} \frac{2a(\xi)}{b^2(\xi)}d\xi \right)dy; \quad  S([\xi_1,\xi_2])= S(\xi_2)-S(\xi_1) \\
			m(x) &= \frac{1}{b^2(x)s(x)}, \hspace{3.5cm}	M(dx)= m(x)dx,
		\end{split}
	\end{equation*}
where $$s(x)=S^\prime(x)=\exp\left( -\int^x_{x_0} \frac{2a(y)}{b^2(y)}dy \right).$$
 $S$ is called the \emph{scale measure} and $m$ the \emph{speed density}.
\end{definition}
Note\cite{1999_Revuz_yor} that the scale function is such that
$
\mathcal{A}s(x)=0,
$
so that $Y_t=S(X_t)$ is a local martingale since:
\begin{equation}
	\label{eq:S(Psi)}
	d Y_t=d S(X_t)= s(X_t)b(X_t)dW_t=\tilde{b}(Y_t) dW_t,
\end{equation}
where $\tilde{b}(\cdot)=s(S^{-1}(\cdot))b(S^{-1}(\cdot))$.
Furthermore, the Markov transition kernel associated with the process $X$ is absolutely continuous with respect to the speed measure, i.e.
$$
P(t,x,dy)=P(X_t\in dy|X_0=x)=\int p(t,x,y)M(dx)=\int p(t,x,y)m(x)dx.
$$
Furthermore \cite{1999_Revuz_yor,2000_Roger_williams,1981_karlin_taylor}, let $\tau_x=\inf\{t \ge 0: X_t=x\}$ be the first passage time that $X$ reaches $x \in (0,\eta)$; then, one can express the probability that $X$ reaches $u_2$ before $u_1$, for $ 0<u_1<x<u_2<\eta$, as a function of the scale measure in the following way:
\begin{equation*}
u(x) = P\left( \tau_{u_2} < \tau_{u_1}|X_0=x\right)=
  \frac{S(x)-S(u_1)}{S(u_2)-S(u_1)}.
\end{equation*}
This means that there is a non trivial link between the scale function, the infinitesimal transition of the process and the speed measure, in particular in terms of the passage times. Therefore, the boundary point classification can be expressed through such mathematical objects. We focus our attention on the \textit{entrance boundary}, the class of points such that if the initial point $X_0$ lies in the interior of the domain they are never reached by the process, while if $X_0$ belongs to it then the process is immediately pushed into the interior of the domain. The formal definition may be given in terms of scale and speed measures: a boundary point is an entrance point if it is possible for the process to go infinitesimally close to it, but it never reaches it, \cite{1999_Revuz_yor,2000_Roger_williams,1981_karlin_taylor}.

\begin{definition}
	Let $x\in [l,r]\subseteq(0,\eta)$.
The right boundary  $r$ is an \emph{entrance boundary} if
\begin{equation}
	\label{entranceboundaryconditions2}
	S(l,x]=\lim_{a \to l} \int_a^x s(v)dv=\infty,  \qquad  \qquad  M(l,x]=\lim_{a \to l} \int_a^x m(v)dv<\infty 
\end{equation}
The left boundary $l$ is an \emph{entrance boundary} if
\begin{equation}
	\label{entranceboundaryconditions}
	S[x,r)=\lim_{a \to r} \int_x^a s(v)dv=\infty, \qquad  \qquad M[x,r)=\lim_{a \to r} \int_x^a m(v)dv<\infty 
\end{equation}
\end{definition}
\smallskip

The following result gives conditions on the parameters to ensure that $0, \eta$ are entrance boundaries points for our process \eqref{eq:SDE_for_SO2}:

 \begin{proposition}\label{prop:boundedness_SDE}
Let  $X=(X_t)_{t\in[0,T]}$  be the solution of the Pearson diffusion \eqref{eq:SDE_for_SO2}. Then, given  the parameter $\alpha,\gamma,\sigma^2,\eta$ we define 
\begin{equation*}\label{eq:condition_SDE_parameters2}
 \nu_1 = \frac{2\alpha\gamma}{\sigma^2\eta}, \quad  \nu_2 = \frac{2\alpha(\eta-\gamma)}{\sigma^2\eta}.
\end{equation*} 
If 
\begin{equation}\label{eq:condition_SDE_parameters}
    \nu= \nu_1 \wedge \nu_2 >1,
\end{equation}
then 
both $x=0$ and $x=\eta$ are entrance boundary points. Hence, if $x_0\in [0,\eta]$, the trajectories of $X$ stay in $(0,\eta)$ with probability 1.
Furthermore,   the following bound holds, for any $T\in \mathbb R_+$ and any $p<\nu$
\begin{equation}\label{eq:bound_inverse_moments}
    \left(\sup\limits_{t\in[0,T]} \mathbb E\left[ \Psi_t^{-p}\right]\vee \sup\limits_{t\in[0,T]} \mathbb E\left[ (\eta -\Psi_t)^{-w}\right]\right)\le C_{T,p}.
\end{equation}
\end{proposition}
 \begin{proof}
After a little algebra, if   $\nu_1$ and ${\nu_2}$ are defined as in \eqref{eq:condition_SDE_parameters2}, given $x_0 \in (0,\eta)$ one obtain that the scale function is
\begin{equation}\label{eq:scale_function_pearson}
    s(x)= \frac{x_0^{\nu_1}(\eta-x_0)^{\nu_2}}{x^{\nu_1}(\eta-x)^{\nu_2}} ,
\end{equation}
while the speed measure density is
\begin{equation}\label{eq:speed_density_pearson} 
    m(x)  =   \frac{x^{{\nu_1}-1}(\eta-x)^{{\nu_2}-1}}{\sigma^2 x_0^{\nu_1}(\eta-x_0)^{\nu_2}} .
\end{equation} 
We need to prove that conditions \eqref{entranceboundaryconditions2}-\eqref{entranceboundaryconditions} hold with $l=0$ and $r=\eta$. Such conditions are closely related to the integrability of the functions \eqref{eq:scale_function_pearson}-\eqref{eq:speed_density_pearson} close to the boundary. More precisely, condition \eqref{entranceboundaryconditions2} for $x=0$ being entrance boundary is equivalent to the integrability of  $m(x)$ and the explosion of \eqref{eq:scale_function_pearson}: those happen whenever $\nu_1>1$, i.e. first condition of \eqref{eq:condition_SDE_parameters} is satisfied. The second of \eqref{eq:condition_SDE_parameters}, implies again that conditions \eqref{entranceboundaryconditions2} are satisfied around $x=\eta$.  Hence, the boundary points $0,\eta$ are entrance boundaries, and the thesis is achieved.  
 The bound \eqref{eq:bound_inverse_moments} of the inverse moments may be found in   \cite{2008_Sorensen_forman}.
	\end{proof}
\begin{remark} 
From \eqref{eq:condition_SDE_parameters}, in particular $\nu_2>1$, we get  $\eta- \gamma \geq 0$; hence, naturally the mean-reverting level lies in the interval $(0,\eta)$.
\end{remark}

\medskip

For completeness, we stress that the speed and scale functions are involved in the definition of the invariant measure for the equation \eqref{eq:SDE_for_SO2}. Indeed,  its density, when it exists, is given by
 \begin{equation*}%\label{eq:density_invariant_measure_speed}
	\widetilde{p}(x)=c_1 \frac{S(x)}{s(x)b^2(x)}+c_2 \frac{1}{s(x)b^2(x)}=c_1  S(x) +c_2 m(x),
 \end{equation*}
 where  $c_1$ and $c_2$ are such that $\widetilde{p}$ is positive and it is the density of a probability measure. In our specific case, said $\Gamma$ the Gamma function, we get that the invariant density has the form \begin{equation}\label{eq:invariantmeasure}
	\widetilde{p}(x)=\frac{1}{\eta^{{\nu_1}+{\nu_2}-2}}\frac{\Gamma({\nu_1}+{\nu_2})}{\Gamma({\nu_1})\Gamma({\nu_2})}x^{{\nu_1}-1}(\eta-x)^{{\nu_2}-1}1_{[0,\eta]}(x),
 \end{equation} 
which is the density of a $Beta(\nu_1,\nu_2)$ with support in $[0,\eta].$
 \medskip
 
In conclusion, we consider a bounded noise addeded at the boundary of the system \eqref{eq:modello_Natalini}, even though it is generated by means of Brownian motion $W$, which is an unbounded stochastic process. Furthermore, as time increases, it converges to a process with invariant distribution.

\subsection{The Lamperti Transform: towards a constant diffusion}

\medskip
The process transform \eqref{eq:S(Psi)} produces a diffusion process with zero drift. In this section we consider a different transform with the aim to obtain a new stochastic process with a constant diffusion coefficient, so that we may bypass the issues typically connected to a state-dependent diffusion coefficient.
A typical way to obtain this goal is to perform a time-change for the Brownian motion, \cite{2021_Capasso}. Since by considering a transformation $Y=F(\Psi)$ of the solution of equation \eqref{eq:SDE_for_SO2} and applying It\^o formula  we get 
\begin{align}\label{eq:ito lamperti}
  dY_t & =  
    \left( F^\prime(\Psi)a(\Psi_t)+ \frac{1}{2} F^{\prime\prime}(\Psi_t)b^2(\Psi_t)	\right) dt + F^\prime(\Psi)b(\Psi_t) dW_t,
\end{align} a natural  transform of the diffusion coefficient $b$ defined by  \eqref{eq:diffusion_Pearson} is the following
\begin{align*}
	y=F(z) & =   \sigma\int^z \frac{du}{ b(u)} =\int^z \frac{du}{\sqrt{u(\eta-u)}}  
	= {2}\arcsin\biggl(\sqrt{\frac{{z}}{{\eta}}}\biggr).
\end{align*}

\medskip
Hence, we apply the so-called \emph{Lamperti transform}
\begin{equation}\label{eq:Lamperti_trasform}
  Y=F(\Psi)={2}\arcsin\left(\sqrt{\frac{\Psi}{\eta}}\right),
\end{equation}
which is well defined since by Proposition \eqref{prop:boundedness_SDE}  we have that $\Psi\in(0,\eta)$. The process $\Psi$ can be recovered by the anti-transform
\begin{equation}\label{eq:Lamperti_anti_trasform}
	\Psi = F^{-1}(Y) = \eta\sin^2\biggl(\frac{ Y}{2}\biggr).
\end{equation}
The process \eqref{eq:Lamperti_trasform} is indeed a solution of a SDE with constant diffusion $\sigma$.

\begin{proposition}
Let $\Psi \in (0,\eta)$ be the unique pathwise solution of the SDE \eqref{eq:SDE_for_SO2}, such that condition \eqref{eq:condition_SDE_parameters} is satisfied. 	For $t\in[0,T], T\in\mathbb R_+$, the process $Y_t=F(\Psi_t)\in (0,\pi)$, with $F$ given by \eqref{eq:Lamperti_trasform} is solution  of the following SDE
	\begin{equation}\label{eq:Lamperti_equation}
		d Y_t = \left[ a_1 \cot\left(\frac{ Y_t}{2}\right)- a_2 \tan\left(\frac{ Y_t}{2}\right)\right]dt+ \sigma dW_t.
	\end{equation}
with $a_1, a_2$ are positive constants given by
\begin{equation}\label{eq:Lamperti_SDE_coefficients}
a_1=\frac{4\alpha\gamma-\sigma^2\eta}{4\eta}, \qquad a_2 =\frac{4\alpha(\eta- \gamma)-\sigma^2\eta}{4\eta}.
\end{equation}
Furthermore, there is a unique point \begin{equation}\label{eq:Lamperti_equation_x*}y^*=
2\arctan\left(\sqrt{\frac{4\alpha\gamma-\sigma^2\eta}{4\alpha(\eta- \gamma)-\sigma^2\eta  }}\right)\end{equation} such that 
the non Lipschitz drift
\begin{equation}\label{eq:drift_post_lamperti}
f(y)=  a_1 \cot\left(\frac{ y}{2}\right)- a_2 \tan\left(\frac{ y}{2}\right),
\end{equation}
satisfies $f(y^*) =0.$
\end{proposition}
\begin{proof}
From \eqref{eq:ito lamperti} and \eqref{eq:Lamperti_trasform}, it is trivial that the new diffusion coefficient is $\tilde{\sigma}=\sigma$, while the new drift $\tilde{a}$ has to be calculated. For any $t$ and $Y_t\in (0,\pi)$
	\begin{align*}
	\tilde{a}\left(Y_t\right)&= 	\frac{\sigma a(F^{-1}(Y_t))}{b(F^{-1}(Y_t))}  - \frac{\sigma}{2}  b^\prime(F^{-1}(Y_t))   \\
		& = \frac{\alpha(\gamma-(F^{-1}(Y)))}{\sqrt{F^{-1}(Y)(\eta-(F^{-1}(Y)))}}-\frac{\sigma^2}{4}\frac{\eta-2(F^{-1}(Y))}{\sqrt{F^{-1}(Y)(\eta-(F^{-1}(Y)))}}\\
		& = \frac{\alpha(\gamma-\eta\sin^2\left(\frac{Y}{2})\right)}{\sqrt{\eta\sin^2\left(\frac{Y}{2}\right)\left(\eta-\eta\sin^2\left(\frac{Y}{2}\right)\right)}}
  -\frac{\sigma^2}{4} 
  \frac{\eta-2\eta\sin^2\left(\frac{Y}{2}\right)}{\sqrt{\eta\sin^2\left(\frac{Y}{2}\right)\left(\eta-\eta\sin^2\left(\frac{Y}{2}\right)\right)}}
\end{align*}
After a little algebra, equation \eqref{eq:Lamperti_equation} with parameters \eqref{eq:Lamperti_SDE_coefficients}, can be achieved. The existence and pathwise uniqueness of the solution to \eqref{eq:Lamperti_equation} derive from the existence and uniqueness of the solution to \eqref{eq:SDE_for_SO2}, due to the monotonicity of the Lamperti transform. The existence of $y^*$ is trivial. \end{proof}
From \eqref{eq:drift_post_lamperti} one can easily show the following monotonicity result.
\begin{proposition}\label{prop:properties_f_post_Lamperti}
The drift \eqref{eq:drift_post_lamperti} shows a superlinear growth in $(0,\pi)$ and for any $y \in [0,\pi]$ \begin{equation}
\label{eq:derivata_negativa_f_lamperti}   f^\prime(y) \leq -C_0,
\end{equation}
where $C_0= (2\alpha- {\sigma}^2)/{4}>0.$
\end{proposition}
\begin{proof}
The inequality \eqref{eq:derivata_negativa_f_lamperti} is obtained simply by differentiating the \eqref{eq:drift_post_lamperti}. The condition  $C_0>0$  derives easily from condition \eqref{eq:condition_SDE_parameters}.
\end{proof}
Thus, the Lamperti transform is a useful tool to overcome the issue of dealing with an SDE \eqref{eq:SDE_for_SO2} with a not Lipschitz diffusion coefficient. Indeed, in the case of not Lipschitz diffusion coefficient the standard numerical approximation of the equation are not able to preserve the boundaries. This implies that monotonicity and stability problems can arise.

\subsection{Numerical approximation of the boundary condition:Backward Euler-Maruyama scheme}

We are interested in the numerical approximation of the SDE \eqref{eq:SDE_for_SO2} in $[0,\eta]$. Our goal is to construct explicit numerical schemes preserving the properties of the SDE. It is well known that the Eulero-Maruyama scheme does not preserve the boundary of the domain since, in the approximation close to the boundary, the Wiener increment at the next time step could be large enough to force the solution out of the interval $(0,\eta)$, \cite{1992_Kloeden_Platen,2011_Dangerfield_boundary preserving}. The boundedness of the numerical solution can be maintained by adding conditions for absorption or reflection at the boundary, \cite{2010_Lord_Koekkoek_Dijk}, but these conditions may introduce a bias into the numerical solution. To overcome such a problem,  we first apply the Lamperti transform \eqref{eq:Lamperti_equation}, so that we do not care about the boundary problem for the diffusion coefficient. Then, by \eqref{eq:Lamperti_anti_trasform}, we recover the original discretised stochastic process. 
In order to consider a numerical scheme for the additive model \eqref{eq:Lamperti_equation}, let us notice that the drift \eqref{eq:drift_post_lamperti} is only one-sided Lipschitz, periodic decreasing in $(0,\pi)$, as shown by Proposition \ref{prop:properties_f_post_Lamperti}; a regularization is necessary both to preserve the monotonicity and to avoid the production of divergence solutions by the Eulero-Maruyama scheme \cite{2012_kloeden,2021_chen_Gan_wang}. 
The properties of the drift lead to a natural choice to truncate it and to go beyond the problem of monotonicity.
Smooth truncation is used to overcome
the difficulties caused by the presence of non Lipschitz and super growing drift coefficient, \cite{2012_kloeden}, while the
sloping truncation allows the method to maintain the monotonicity of the
drift, \cite{2021_chen_Gan_wang,2023_chen_LSST}. The property of monotonicity is  important also for the stability of the distribution. 
Given    $
    y^* \in (0,\pi)$ in \eqref{eq:Lamperti_equation_x*}, let   $k>0$ a fixed constant and define
\begin{equation}\label{eq:def_delta_star_SDE}
    \Delta^* = (y^* \wedge (\pi-y^*) \wedge 1)^{\frac{1}{k}}.
\end{equation}
The idea is to smooth the drift close to the boundary via a first-order Taylor expansion and regularize it outside $[0,\pi]$. 
For every $\Delta < \Delta^*$, let us define the following partition of $\mathbb R:$
$$
P=\{P_1,P_2,P_3,P_4,P_5\}=\{(-\infty,0),[0,\Delta^k),[\Delta^k,\pi-\Delta^k], (\pi-\Delta^k, \pi],(\pi,\infty)\}.
$$ 
Therefore, a truncated drift function for the SDE \eqref{eq:Lamperti_equation}, denoted by $f_{\Delta}$ is given by:
 $$   f_{\Delta}(y)  \! := \! \! 
    \begin{cases}
   f(\pi-\Delta^k)+\frac{1}{2} \Delta^k f'(\pi-\Delta^k)-C_0(y-\pi+\frac{1}{2}\Delta^k), \ \ & y\in P_5;  \\
    f(\pi-\Delta^k)+f'(\pi-\Delta^k)(y-\pi+\Delta^k)&\\
    \hspace{1cm}-\frac{1}{2\Delta^k}(f'(\pi-\Delta^k)+C_0)(y-\pi+\Delta^k)^2, \ \ & y\in P_4;  \\
  f(x), \ \ & x\in P_3;  \\
  f(\Delta^k)+f'(\Delta^k)(x-\Delta^k)+\frac{1}{2\Delta^k}(f'(\Delta^k)+C_0)(y-\Delta^k)^2, \ \ &    y\in P_2;  \\
  f(\Delta^k)-\frac{1}{2}\Delta^k f'(\Delta^k)-C_0(y-\frac{1}{2}\Delta^k), \ \ & y\in P_1. 
\end{cases}
$$
Figure \ref{fig:f_truncated} shows the effect of the smoothing truncation on the drift $f$ given by \eqref{eq:drift_post_lamperti} for a specific set of parameters and different diffusion coefficients.
\begin{figure}[h]
\centering
\includegraphics[width=0.9\textwidth]{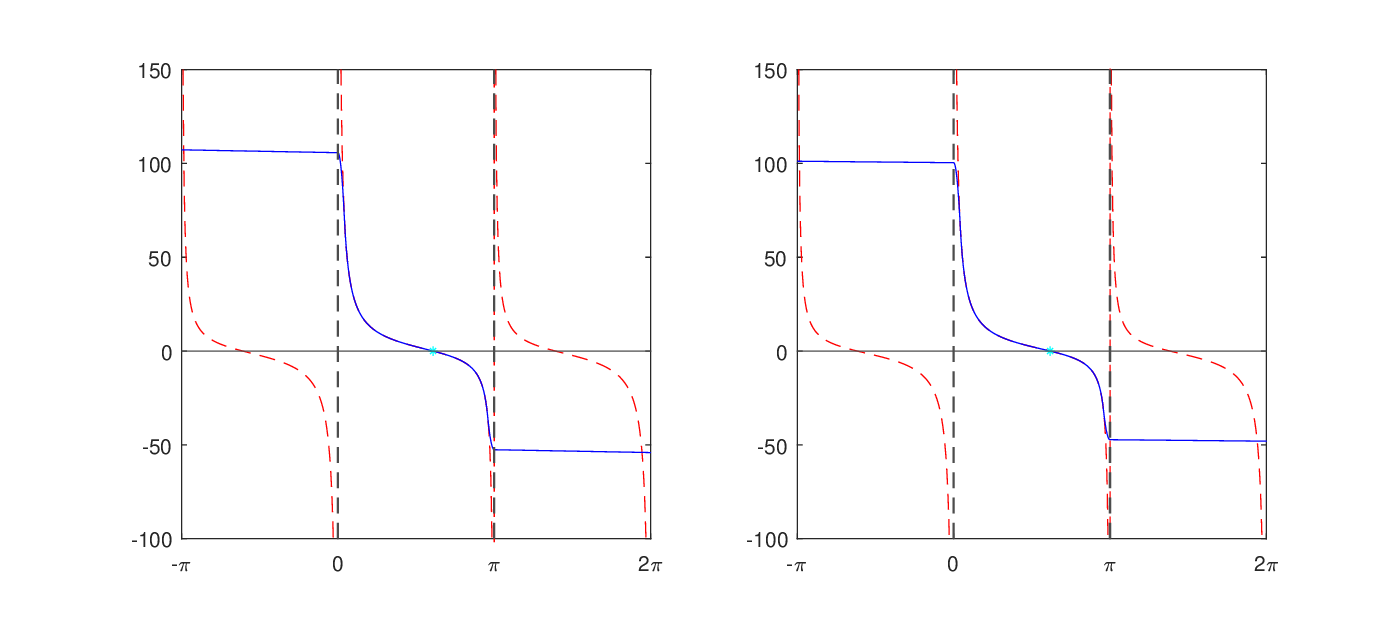}
 {\hspace{3.3cm} (a) \hspace{4.5cm} (b)}

\caption{Drift $f(x)$  of equation \eqref{eq:Lamperti_equation} (dashed line) and the LSST drift $f_{\Delta}(x)$, (solid line) for $\alpha=7, \gamma=1,\eta=1.5, \Delta=10^{-4}, k=0.22$. (a) $\sigma_1=0.25$, $x^*=1.912$ , $C_0=0.48$; (b) $\sigma_3=1$, $x^*=1.938$, $C_0=0.25$.}\label{fig:f_truncated}
\end{figure}

In $[0,\pi]$ function $f$ is monotone decreasing, periodic of period $\pi$ and super growing to infinity and minus infinity close to $0$ and $\pi$;   $f_\Delta$ is smoothly truncated close to the boundary and regularized outside, and monotonicity is preserved. We see that for the two different diffusion coefficients, $f_\delta$ does not change too much; only the slop of the regularized part outside $[0,\pi]$ is a bit larger for $\sigma=0.25$.
The new drift $f_\Delta$ is Lipschitz continuous and one-side Lipschitz continuous. Furthermore,  the drift $f_\Delta$ is a good approximation of the continuous drift $f$. Indeed, the following result is proven in    \cite{2023_chen_LSST}.

\newpage

\begin{proposition}
    Suppose that the assumption \eqref{eq:condition_SDE_parameters} holds. Let $0 < k < 1$ and $\Delta < \Delta^*$ be given. Then, for any $x  \in [0,\eta]$
    $$f^\prime_\Delta(x)\le -C_0;$$
    and for any $x,y \in [0,\eta]$
\begin{eqnarray*}
    \left| f_\Delta(x)-f_\Delta(y) \right| &\le& \pi^2 C_0 \Delta^{-2k} |x-y|,\\
     (x-y)\left( f_\Delta(x)-f_\Delta(y) \right) &\le&   -C_0   |x-y|^2.
\end{eqnarray*}
Furthermore, if $\nu>3$, where $\nu$ is defined in \eqref{eq:condition_SDE_parameters}, for any $1<p<\nu$, and $\kappa_1=2k(p-1)$
\begin{equation*}
\sup_{t\in[0,T]}\mathbb E\left[\left| f(Y_t)-f_\Delta(Y_t)\right|^2\right]\le C_{T,p} \Delta^{\kappa_1}.
\end{equation*}
\end{proposition} 

\medskip

The overall numerical procedure is called the \emph{Lamperti Sloping Smooth Truncation} (LSST). Let us consider an equi-partition of $[0,T]$ 
\begin{equation}\label{eq:time_partition}  \Xi_t =  [0=t_0, t_1, \dots, t_N=T], \end{equation} 
and set $ \Delta_t =\frac{T}{N}$, $N\in \mathbb N$. 
We adopt the following notation for a general function $u$ evaluated at the mesh nodes, $u^n=u(t_n)$, for any $n\in \{0,\ldots,N\}.$ 
 
\medskip
The sampling scheme of the stochastic boundary condition $\Psi$ solution of equation \eqref{eq:SDE_for_SO2} is obtained by numerically sampling $(Y,\Psi)$, solutions of \eqref{eq:Lamperti_equation} and \eqref {eq:Lamperti_anti_trasform}, respectively by adopting the LSST procedure that reads as follows: given  $\Psi_0=\psi^0$, set  $y^0={2}\arcsin\left(\sqrt{ {\psi^0}/{\eta}}\right)$ and, for $n=1,\dots,N$

\smallskip

\begin{equation}\label{eq:LSTT_scheme}
    \begin{split}
 \mbox{\emph{LSST Scheme: \hspace{1cm}}}       y^{n} & = y^{n-1} + f_\Delta(y^{n-1})\Delta_t   + \sigma \Delta W^{n}, \hspace{2.3cm} \\
        \psi^{n} &  = \sin^2\left( {y^{n}}/{2}\right),
    \end{split}
\end{equation}
where $\Delta W^j=W^j-W^{j-1}, \;j=1,\dots, N$ are   independent Brownian increments.

\medskip

Given the LSST scheme \eqref{eq:LSTT_scheme}, let $ {y}_t, \bar{\psi}_t$ be the time continuous LSST approximation (TC-LSST)   defined for any $t\in[0,T]$ as follows. Let  $\bar{f}_\Delta$ be the function defined as 
$$
\bar{f}_\Delta(y_t)=f_\Delta(y^n), \quad  t\in [t_n,t_{n+1}).
$$
The time continuous LSST approximation of the scheme \eqref{eq:LSTT_scheme}, given $W=(W_t)_{t\in [0,T]}$ a Wiener process, is provided by 
 \begin{equation}\label{eq:LSTT_scheme_time_continous}
    \begin{split}
        \mbox{\emph{TC-LSST Scheme: \hspace{1cm}}}        {y}_t & = {y}^0+ \int_0^t \bar{f}_\Delta(y_t)dt + \sigma W_t,  \hspace{2.3cm}\\
         {\psi}_t&  = \sin^2\left( {{y}_t }/{2}\right).
    \end{split}
\end{equation}

 \medskip
 
In   \cite{2023_chen_LSST} it is proven that the LSST 
numerical approximation has not only strong convergence, but the rate of convergence, under specific assumptions on the coefficients, may go up to one.
 
\begin{proposition}\label{prop:convergence_SDE} 
Suppose $\nu> 3$ and take  $p, k, \bar{r}$ in the ranges according to the  following cases

\begin{center}
\begin{tabular}{c| c c c c }
    Case & $\nu$ & $p$ & $k$  &$\bar{r}$\\ 
    & & & &   \\[-1em]
    \hline
    Case 1 &  $(3,13/4]$& $(3,\nu)$& $1/4$& 1\\
    & & & &   \\[-1em]
    Case 2 & $( 13/4,\infty)$&$(\sqrt{13 \nu}/2,\nu)$&$2p/9\nu$& $ 2k(p-1)\wedge 2$
\end{tabular}
 \end{center}
 
Then, for any $\Delta_t < \Delta^*$, the TC-LSST scheme $( {y}_t, {\psi}_t)_{t\in[0,1]}$ defined by \eqref{eq:LSTT_scheme_time_continous}  have the
following strong convergence
\begin{equation*}
    \sup_{t\in[0,T]} \mathbb E\left[|Y_t-y_t|^2\right],\sup_{t\in[0,T]} \mathbb E\left[|\Psi_t- {\psi}_t|^2\right]\le C_{T,p} \Delta^{\bar{r}}.
\end{equation*}\label{prop:LSTT_convergence}
\end{proposition}
\begin{remark}\label{remark:concergence_order2}
    Note that whenever $\nu>{11}/2$, then $\bar{r}=2$;  hence, the order of convergence is always one.
\end{remark}

Let us consider the Pearson process $\Psi=\{\Psi_t\}_{t\in[0,1]}$ in the range $[0,\eta]=[0,1.5],$ with mean reverting level $\gamma=1$ and rate $\alpha=7$.  In the LSST scheme \eqref{eq:LSTT_scheme} we take $\Delta_t=2^{-15}$ and $k=0.22$. 
We consider two experiments with diffusion coefficients $\sigma_1=0.25$ and $\sigma_3=1$, so  that $\nu=\nu_1=74.67$ and $\nu=\nu_2=4.67$, respectively.  Hence,  for the first parameters configuration, we get that the condition $\nu=\nu_1>{11}/2=5.5$ of Remark \ref{remark:concergence_order2} is satisfied; we expect strong convergence of order one in the sense of \eqref{eq:second_order_error}. In the second case, we have $\bar{r}\in\left(1.27,1.6 \right)$, so we expect an order of convergence in $\left(0.637,0.80\right)$.
We have performed numerical experiments  and the approximation errors are tested both  in terms of the $L^2(\Omega)$-norm at  the final time $T=1$, i.e.
 \begin{equation}\label{eq:second_order_error}
     e_{2,T,\Delta}=(\mathbb{E} [|Y_T - y_T|^2])^{\frac{1}{2}},
 \end{equation}  
 and in terms of the  time uniform  $L^2(\Omega)$ norm along the discretization partition 
\begin{equation}
\label{eq:second_order_error_uniform}
     \epsilon_{2,\Delta}=\sup_{t_n\in \Xi_t}(\mathbb{E} [|Y_{t_n} - y_{t_n}|^2])^{\frac{1}{2}},
 \end{equation} 
 where $\Xi_t$ is the time mesh in \eqref{eq:time_partition}.
The \emph{true} solution $Y_T$ is identified as the LSST approximation with the finer step-size $\delta=2^{-15}$. 
 We define  by $(y_{j}^n)_{n=0,...,T/{\Delta_t}}$   the discrete $j$-th trajectory out of $N\in \mathbb N$ given  by \eqref{eq:LSTT_scheme}, for any  larger $\Delta_t  \in \delta\cdot [2^{4}, 2^{5}, 2^{6}, 2^{7}, 2^{8}]$.   In particular, we use trajectory samples of dimension $N=10000$ and estimate for any $\Delta_t$ the error \eqref{eq:second_order_error} by means of the sample mean
 \begin{equation}
     \label{eq:second_order_error_estimate_final_time}
     \hat{e}_{{2,T,\Delta}}^2=\frac{1}{N}\sum_{k=1}^N\left|Y_T - y_{j}^{T/\Delta} \right|^2,
 \end{equation}
 while the time uniform error \eqref{eq:second_order_error_uniform} is estimated by 
 \begin{equation} 
     \label{eq:second_order_error_estimate}
   \hat{\epsilon}_{2,\Delta}^2=\sup_{t_n\in \Xi_t}\frac{1}{N}\sum_{j=1}^N\left|   Y_{t_n} - y_{j}^{n}\right|^2.
 \end{equation} 
   
By assuming a power law  for the error $e$, i.e $e  = C \Delta ^q$, where $q=\bar{r}/2$, where $\bar{r}$ is introduced in Proposition   \eqref{prop:LSTT_convergence},  since $\log e = \log C + q\log\Delta_t$ we estimate the order of strong convergence  $q=\bar{r}/2$ 
by means of a least squares fit, given the error sample $\left(\Delta,\hat{e}\right)_\Delta$.  
For $\hat{e}= \hat{e}_{{2,T,\Delta}}$  the estimates of $q$ are $1.016$ and $ 1.022$ for  $\sigma_1=0.25$ and $\sigma_2=1$, respectively. Experimental estimates are shown in Table \ref{table:SDE_error_T}, and Figure \ref{fig:combined_error} depitches them. Precisely, 
 Figure \ref{fig:combined_error} (first row) shows the points $ \left(\Delta,\hat{e}_{{2,T,\Delta}}\right)_\Delta$ in the log-log  scale for (a) $\sigma_1=0.25$ and (b) $\sigma_2=1$; for $\hat{e}= \hat{e}_{{2,\Delta}}$  the estimates of $q$ are  $0.97$ and $0.51$, respectively.  Figure \ref{fig:combined_error} (second row) shows the points $ \left(\Delta,\hat{e}_{{2,\Delta}}\right)_\Delta$ in the log-log  scale for (c) $\sigma_1=0.25$ and (d) $\sigma_2=1$. Hence, evidence of first order strong convergence of  LSST scheme is experimentally proven.    \begin{figure}[h!]
    \centering
    \begin{subfigure}[b]{0.4\textwidth}
        \centering
        \includegraphics[width=\textwidth]{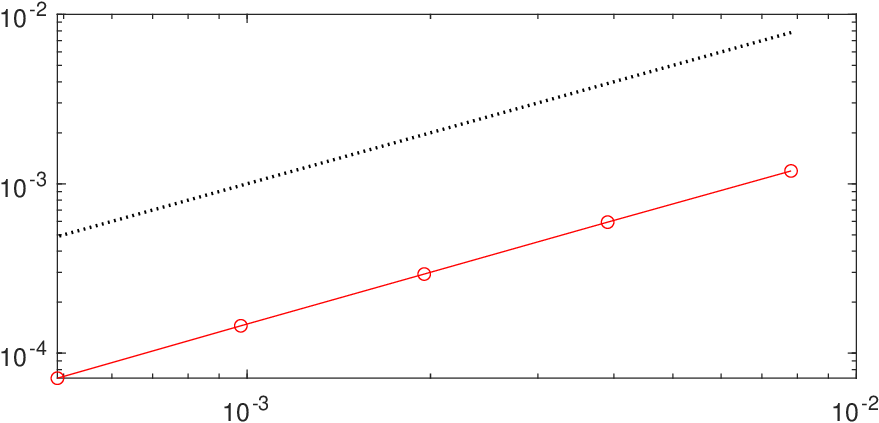}
        \caption{} 
    \end{subfigure} 
    \begin{subfigure}[b]{0.4\textwidth}
        \centering
        \includegraphics[width=\textwidth]{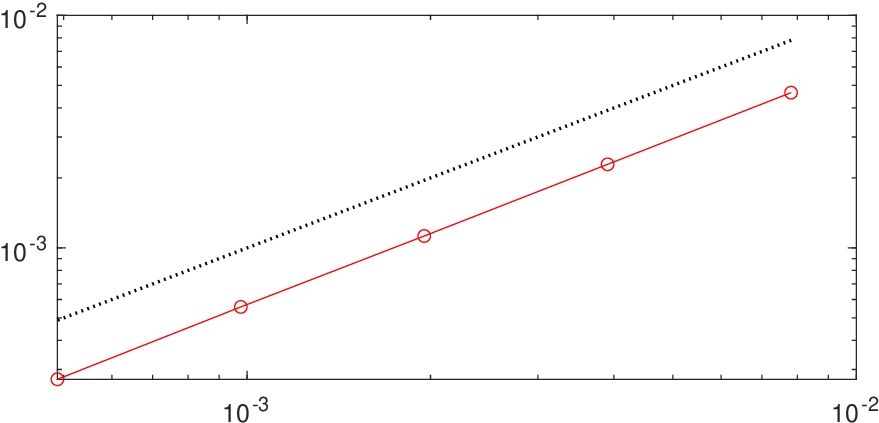}
        \caption{} 
\end{subfigure}
\begin{subfigure}[b]{0.4\textwidth}
        \centering
        \includegraphics[width=\textwidth]{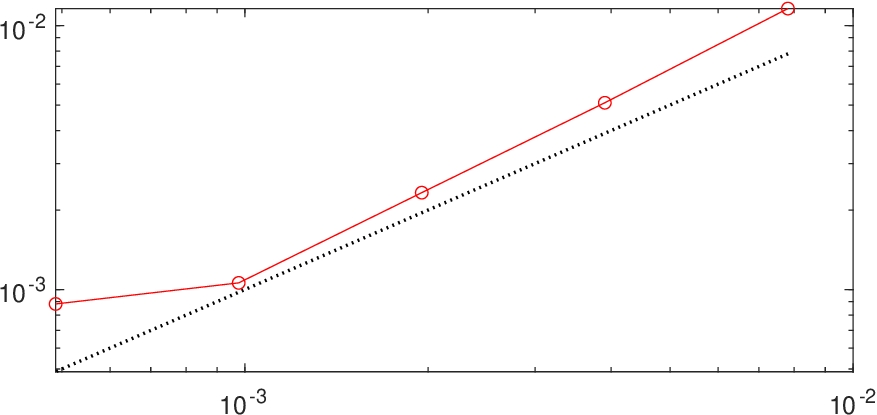}
        \caption{ } 
    \end{subfigure}
    \begin{subfigure}[b]{0.4\textwidth}
        \centering
        \includegraphics[width=\textwidth]{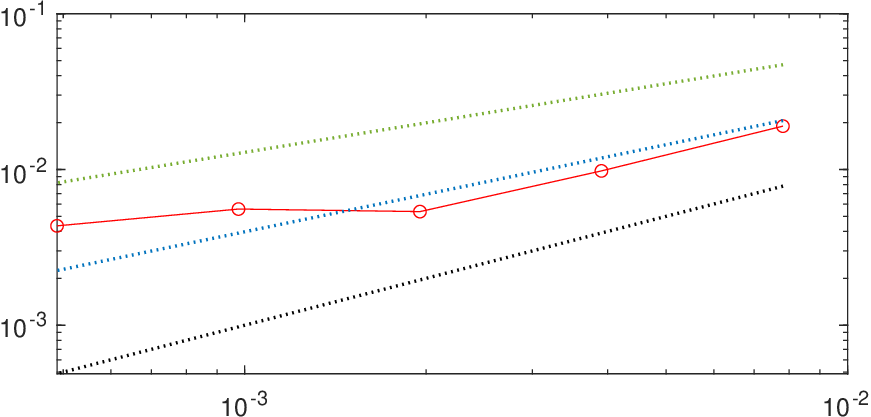}
        \caption{ }% 
\end{subfigure}%\label{fig:errore_uniforme_set2}
\caption{Log-log   of the estimations of the $L^2$-errors at time $T$ $\left(\Delta,\hat{e}_{{2,T,\Delta}}\right)_\Delta$  (solid line-first raw), and the uniform one $\left(\Delta,\hat{\epsilon}_{{2,\Delta}}\right)_\Delta$ (solid line-second row). Parameters are:  $\alpha=7,\gamma=1,\eta=1.5$ and different diffusion coefficients:   (a)-(c) $\sigma_1=0.25$; (b)-(d) $\sigma_3=1$. Dashed black line: slope reference 1; dashed blue line: slope reference 0.8; dashed green line: slop reference 0.6.}
\label{fig:combined_error}
\end{figure}

 \begin{table}[h!]
    \centering  
    \begin{subtable}{0.5\textwidth}
    \centering
        \begin{tabular}{|c|| c||c|}        
            \hline 
            & \multicolumn{1}{|c||}{$\sigma_1=0.25$} & \multicolumn{1}{|c|}{$\sigma_3=1$} \\
            \hline\hline
            ${\Delta}/{\delta}$ & $\hat{e}_{{2,T,\Delta}}$    & $\hat{e}_{{2,T,\Delta}}$ \\
            \hline
            $2^{8}$   & 1.19E-03 & 4.65E-03  \\
            $2^{7}$    &  5.94E-04&	2.29E-03 \\        
            $2^{6}$    &  2.93E-04&	1.13E-03 \\        
            $2^{5}$    &  1.45E-04&	5.57E-04 \\
            $2^{4}$   &   7,12E-05&	2.72E-04 \\
            \hline
        \end{tabular} \caption{}
         \end{subtable}
        \begin{subtable}{0.49\textwidth} \centering
          \begin{tabular}{|c|| c||c|}        
            \hline 
            & \multicolumn{1}{|c||}{$\sigma_1=0.25$} & \multicolumn{1}{|c|}{$\sigma_3=1$} \\
            \hline\hline
            ${\Delta}/{\delta}$ & $\hat{\epsilon}_{{2,\Delta}}$    & $\hat{\epsilon}_{{2,\Delta}}$ \\
            \hline
            $2^{8}$   &  1.16E-02&	1.90E-02   \\
            $2^{7}$    &  5.10E-03&	9.79E-03  \\        
            $2^{6}$    & 2.33E-03	&5.36E-03  \\        
            $2^{5}$   &    1.06E-03&	5.57E-03   \\
            $2^{4}$    &   8.82E-04&	4.35E-03  \\
      \hline
        \end{tabular}
        \caption{}
         \end{subtable}
          \caption{Convergence error estimates   for $\alpha=7,\gamma=1,\eta=1.5$ and different diffusion coefficients:    $\sigma_1=0.25$ and $\sigma_3=1$. (a) $L^2(\Omega)$  approximation error at time $T$: estimation \eqref{eq:second_order_error_estimate_final_time}. (b) Time uniform  $L^2(\Omega)$  approximation error: estimation \eqref{eq:second_order_error_estimate}. }\label{table:SDE_error_T}
    \end{table}

Even if conditions in Case 1 of Proposition \ref{prop:LSTT_convergence} are fulfilled, the convergence rate may be equal to one, since Proposition \ref{prop:LSTT_convergence}  gives only some sufficient conditions, \cite{2023_chen_LSST}. Indeed, we perform an experiment with the following  data: 
$\alpha=3.9,\gamma=0.9,\eta=1.5$ and  $\sigma_3=1$; this parameter configuration is such that $3<\nu<13/4,$ and thus, accordingly to Proposition \ref{prop:LSTT_convergence}, we would expect a convergence rate equal to 1/2 but actually from Figure \ref{fig:errore_log_uniforme_set3}-(a) we observe that the $L^2$ convergence rate at final time may be equal to 1, while the uniform error rate as shown in Figure \ref{fig:errore_log_uniforme_set3}-(b) is closed to 0.5.

\begin{figure}[h] 
\centering
\begin{subfigure}[b]{0.4\textwidth}
        \centering
        \includegraphics[width=\textwidth]{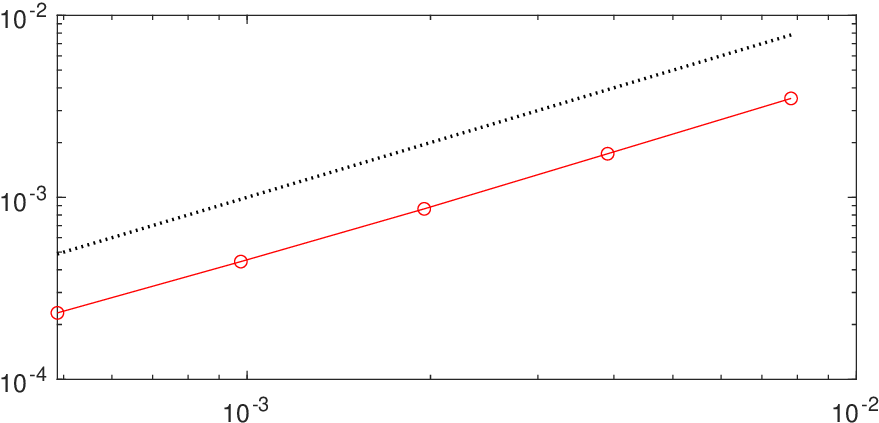}
        \caption{ }  
    \end{subfigure}
    \begin{subfigure}[b]{0.4\textwidth}
        \centering
        \includegraphics[width=\textwidth]{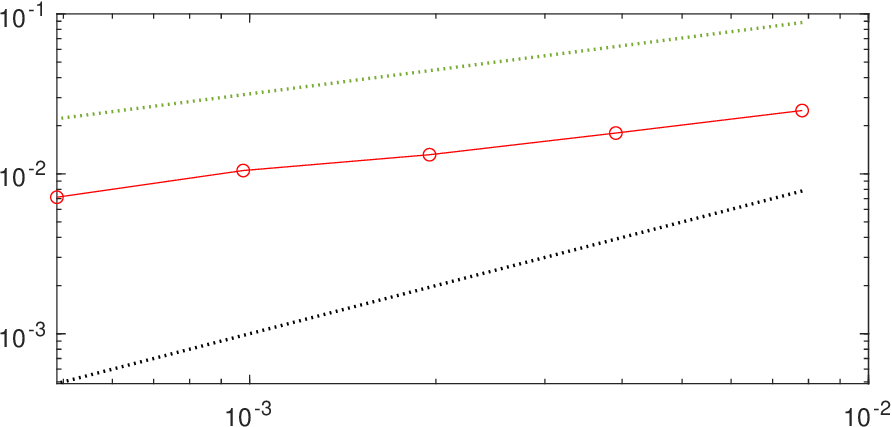}
        \caption{ }
\end{subfigure}
\caption{Log-log representation of the estimations of the $L^2$-errors at time $T$ $\left(\Delta,\hat{e}_{{2,T,\Delta}}\right)_\Delta$  (a), and the uniform one $\left(\Delta,\hat{\epsilon}_{{2,\Delta}}\right)_\Delta$ (b). Parameters are:  $\alpha=3.9,\gamma=0.9,\eta=1.5$ and  $\sigma_3=1$. Dashed black line: slope reference 1;  dashed green line: slop reference 0.5.}
\label{fig:errore_log_uniforme_set3}
\end{figure}

\bigskip

\begin{figure}[h!]
\begin{center}
\includegraphics[width=0.49\textwidth]{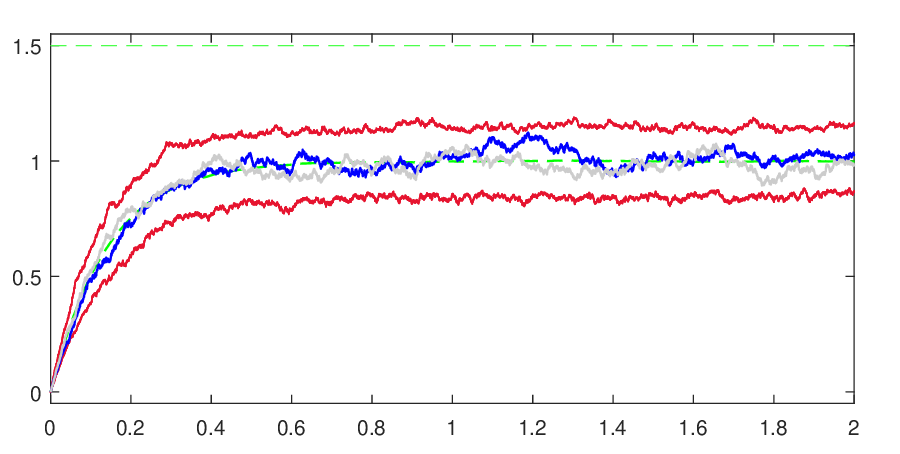}  \includegraphics[width=0.49\textwidth]{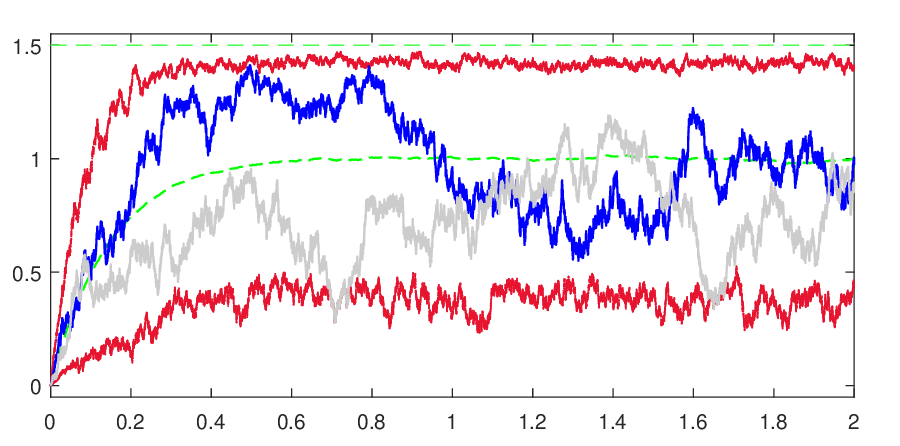} 
\end{center}
 {\hspace{2.9cm}(a) \hspace{5.3cm} (b)}
	\caption{Two random trajectories (filled line) kept within the max and the min of all the trajectories (dotted lines); the sampled mean process (dashed centered line) and boundary of the domain (lower and higher dashed lines) $[0,\eta]$ for $\alpha=7,\gamma=1,\eta=1.5$, $\Psi_0=0$ and different diffusion coefficients:   (a) $\sigma_1=0.25$; (b) $\sigma_3=1$. }	\label{fig:min_max_mean}
\end{figure}

 Figure  \ref{fig:min_max_mean} shows some random trajectories of the SDE solution process with small and large diffusion.  
Different values of the diffusion coefficient $\sigma$ are considered to evaluate how the variability of the noise intensity may influence the evolution of the sulphation process. Note that all the parameters satisfy the well-posedness condition \eqref{eq:condition_SDE_parameters}. 
Figure \ref{fig:min_max_mean}-(a) shows how in the small noise configuration ($\sigma_1=0.25$) the process moves around the mean level $\gamma=1$ with low excursions in the domain $(0,\eta)$. Indeed, the minimum and the maximum of the sample at each time are both very far from the boundaries of the domain. A high value of the mean reverting rate $\alpha=7$ together with a low value of the diffusion coefficient induce a lower variability.   For a higher value of the diffusion coefficient such as  $\sigma_3=1$,  Figure \ref{fig:min_max_mean}-(b) shows that the process may instead arrive close to the boundary even though it oscillates around $\gamma=1$ with a wider deviation in the excursion. The high mean reversion helps to return to the average.

\section{A PDE-ODE for modeling sulphation   with stochastic boundary condition}\label{sec:sulfation_stochastic_boundary}

Let us now consider the random model for the sulphation phenomenon, which involves the reaction between sulphur dioxide and calcium carbonate; it is given by the system \eqref{eq:modello_Natalini}-\eqref{eq:modello_Natalini_condition1}  coupled with the stochastic boundary condition \eqref{eq:modello_Natalini_stochastic_boundary} where $\Psi$ is the outdoor local sulphur dioxide density, solution of equation \eqref{eq:SDE_for_SO2}.

\medskip

As in   \cite{2005_GN_NLA,2007_GN_CPDE}, rather than working directly with the local density $\rho$, we work with the porous concentration  $s=\rho/\varphi$, where porosity $\varphi$ is assumed to be a linear function of  calcite density $c$, namely \begin{equation}\label{eq:def_porosity} \varphi(c)=\varphi_1 + \varphi_2 c,
\qquad  \varphi_1,-\varphi_2\in \mathbb R_+.
\end{equation}
We suppose $\varphi_2$ to be negative since the porosity of the material increases when calcite $c$ is transformed into gypsum, that is when $c$ decreases. System \eqref{eq:modello_Natalini} becomes, for $(t,x)\in [0,T]\times \mathbb R_+,$ and $\lambda \in \mathbb R_+$,
\begin{align}
\begin{aligned}\label{eq:PDE} 
\partial_t \left(\varphi(c) s \right) &= \partial_x(\varphi(c) \partial_x s) -\lambda \varphi(c) sc,\\
\partial_t c &= -\lambda \varphi(c) sc.
\end{aligned}
\end{align}
The initial conditions for $s$ and $c$ read
\begin{align}\label{eq:IC} 
s(0,x)=s_0=\frac{\rho_0}{\varphi(c_0)}, \quad c(0,x)=c_0.
\end{align}
The  boundary condition $\rho(t,0)=\Psi_t$ in \eqref{eq:modello_Natalini_stochastic_boundary} becomes
\begin{align}\label{eq:BC} 
s(t,0)=\widetilde{\Psi}_t:=\frac{\Psi_t}{\varphi(c(t,0))}.
\end{align}
From \eqref{eq:BC} and the second equation  in \eqref{eq:PDE}, we easily derive that the boundary condition at $x=0$ for the calcium carbonate is the solution of the following linear ODE
\begin{align*}
\partial_t c(t,0) = -\lambda \Psi_tc(t,0), \qquad c(0,0)=c_0.
\end{align*}
Hence, the stochastic boundary conditions at $x=0$ for $(s,c)$ are given by
\begin{equation}\label{eq:stochastic_boundary_condition_s_c}
\begin{split}
    c(t,0)&=c_0 \exp\left(-\lambda\int_0^t\Psi_u du\right),\\
    s(t,0)&=\widetilde{\Psi}_t = \frac{\Psi_t}{\varphi\left(c_0\exp(-\lambda\int_0^t\Psi_u du)\right)}.
\end{split}
\end{equation}  

\begin{proposition}
Let condition \eqref{eq:condition_SDE_parameters} be satisfied. Then the boundary function   \eqref{eq:stochastic_boundary_condition_s_c} are bounded for any $t\in [0,T].$ More precisely,
\begin{equation}\label{eq:stochastic_boundary_condition_s_c_limit}
\begin{split}
    c(t,0)&<c_0,\\
    0<s(t,0)=\widetilde{\Psi}_t&<\frac{\eta}{\varphi_1+c_0 \varphi_2}=:\widetilde{\eta}
\end{split}
\end{equation}  
\end{proposition}
\begin{proof} The thesis easily derives from Proposition \ref{prop:boundedness_SDE}.
\end{proof}
The random oscillation of $s(t,0)$ corresponds to a poor regularity of the path of $\Psi$: precisely, $\Psi$, and then also $\widetilde{\Psi}$, has $\beta$-H\"older continuous paths, for every $\beta\in (0,1/2)$, but has no better regularity,\cite{2023_MMU}. From the analytical viewpoint, this low regularity is the main feature of our model.
In the deterministic case the well-posedness is achieved for $\widetilde{\Psi} \in W^{1,1}([0,T])$, \cite{2007_GN_CPDE}, while existence and uniqueness in the present case is proven in   \cite{2023_MMU}, as follows.  

\newpage
\begin{theorem}\label{teo:wellposedeness}
 Let conditions \eqref{eq:condition_SDE_parameters} on the parameters of $\Psi$ be satisfied. Then, the trajectories of $\Psi$ are almost surely in   $C^{\beta}$,  with $\beta\in (1/4,\infty)$ and that $\widetilde{\Psi}(0)=0$. Then the system \eqref{eq:PDE}-\eqref{eq:BC}, has a unique solution in the class of $(c,s)\in \left(L^\infty([0,T]\times (0,\infty))\right. $ $\left.\times L^\infty([0,T];W^{1,2}(0,\infty))\right)$.
\end{theorem}
Theorem \ref{teo:wellposedeness} is proven through a splitting argument. Here we state the main construction of the proof since we take advantage of the same splitting strategy to numerically sample the random bivariate process $(c,s)$ in $[0,T]\times\mathbb R_+$.
The random  function $c$ may be found as an explicit solution of the second equation in \eqref{eq:PDE} in terms of $s$, that is
\begin{align}\label{eq:c_formula}
c(t,x) = \frac{\varphi_1 c_0 }{\varphi(c_0 )e^{\lambda\varphi_1\int_0^t s(r,x)dr}- \varphi_2 c_0 },
\end{align}
so that the first equation of system \eqref{eq:PDE}   becomes
\begin{align*}%\label{eq:s}
\partial_t s = \partial_x^2 s +\frac{\varphi_2}{\varphi(c)} \partial_x c\partial_x s -\left(\frac{\varphi_2}{\varphi(c)}\partial_t c+\lambda c\right) s.
\end{align*}
The key point, which is a splitting argument, is to decompose $s$ as the sum
\begin{equation}\label{eq:def_splitted_s}
  s=u+v.  
\end{equation}

\noindent{\bf The independent random heat equation.} The process $u$  in \eqref{eq:def_splitted_s} is the solution of the heat equation with zero initial condition and boundary condition $\widetilde{\Psi}$
\begin{equation*}%\label{eq:u}
\begin{split}
\partial_t u& = \partial_x^2 u,\\  
u(0,x)&=0,\\
u(t,0)&=\widetilde{\Psi}_t.  
\end{split}
\end{equation*}
The function $u$ neglects the reaction and then the nonlinearity, and it handles the randomness and the consequent low regularity of the system. The heat equation is very well understood and a highly oscillating boundary condition becomes smoother in the indoor environment, due to diffusion. We can then use this smoothing effect to study the full system which includes a reaction component.   Indeed, $u\in C^\infty$  in the interior $(t,x)\in (0,T]\times (0,\infty)$ and for $\beta\in (1/4,\infty)$, $u\in L^\infty([0,T];W^{1,2}(0,\infty))$, in particular it is weakly differentiable in $x$; this additional regularity plays an important role in the study of $v$.

\noindent{\bf The random non linear dynamics with deterministic boundary condition.}
The process $v$  in \eqref{eq:def_splitted_s} is the solution of the following equation, defined  for any $(t,x) \in [0,T]\times [0,\infty)$
\begin{equation}\label{eq:v} \begin{split}
\partial_t v & = \partial_x^2 v +\frac{\varphi_2}{\varphi(c)} \partial_x c (\partial_x v +\partial_x u) -\left(\frac{\varphi_2 \partial_t c}{\varphi(c)}+\lambda c\right) (v+u),\\  
v(0,x)&=s_0,\\
v(t,0)&=0, 
\end{split}
\end{equation}
where $c$ is explicitly given by \eqref{eq:c_formula} in terms of $s=u+v$. Thanks to the regularity of $u$  and $\partial_x u$,   existence and uniqueness for $v$ and hence for $s$ is proven in   \cite{2023_MMU}. A nice presentation may be found in    \cite{2024_MACH2023_PDE}.

\section{Discrete approximation  of the random PDE-ODE system}\label{sec:discrete_approximantion}
  
  We introduce a discrete version of the system \eqref{eq:modello_Natalini}-\eqref{eq:SDE_for_SO2} in $[0,T]\times \left[0,\overline{x}\right]$  with $\overline{x}\in \mathbb R_+$ which uses  the splitting strategy discussed in Section \ref{sec:sulfation_stochastic_boundary}.  Let us consider an equi-partition of the space and time intervals with space time meshes  $$\Xi_x \times \Xi_t = \left[0=x_0, x_1, \dots, x_M= \overline{x}\right]\times [0=t_0, t_1, \dots, t_N=T].$$ From now on, we consider the following notations 
  \begin{equation*}%\label{eq:delta}
    \Delta_x = \frac{\overline{x}}{M}\, , \quad \Delta_t=\frac{T}{N}, \quad  { \bar{\Delta}} = \frac{\Delta_t}{\Delta_x ^2}, 
\end{equation*}
with $N\in \mathbb N$, such that $\Delta_t <\Delta^*$, and $\Delta^*$ is as in \eqref{eq:def_delta_star_SDE} and Proposition \ref{prop:convergence_SDE}.
 We adopt the following notation for a general function $g$ evaluated at the mesh nodes $g_m^n=g(x_m,t_n)$, for any $(m,n)\in \{0,\dots,M\}\times\{0,\ldots,N\}.$ 
Since at the simulation level we consider a bounded domain $\left[0,\overline{x}\right]$, we need to take into account a boundary conditions at $\overline{x}$ for function $s=u+v$:  we take the following  Neumann condition for $u$ and $v$ for the right boundary:
\begin{equation}\label{eq:neumann_condition_continuous}\partial_x u(t,\overline{x})=\partial_x v(t,\overline{x})=0.\end{equation}
 We consider the discrete version of the splitted solution \eqref{eq:def_splitted_s} of system \eqref{eq:PDE}, that is for any $(m,n)\in\{1,...,M\}\times\{1,...,N\}$
$$
s^n_m=u^n_m+v^n_m.
$$
According to the theoretical results obtained in   \cite{2023_MMU} it pathwise solves the discrete SDE-PDE system. 

\subsection{The discrete random heat equation: definition and stability.} 
We solve the random PDE system pathwise, i.e. for any given realization $\{\tilde\psi^n\}_{n}(\omega), \omega\in \Omega_P$ of the discrete process 
\begin{equation}\label{eq:psi_tilde_discrete}
    \widetilde\psi^n = \frac{\psi^n}{\varphi\left(c_0\exp(-\lambda\int_0^{t_n}\Psi_u du)\right)}.
\end{equation}
First, we consider a forward time, centered space (FTCS)   approximation  $\{u^n_m\}_{n,m}$  of the heat equation, for $(n,m)\in\{0,\ldots,N-1\}\times\{1,\ldots,M\}$ $$
	\frac{u_m^{n+1}-u_m^n}{\Delta_t}   = \frac{u_{m+1}^n - 2 u_m^n + u_{m-1}^n}{\Delta_x^2}.
 $$  
A slight improvement in computational efficiency
can be obtained with a small rearrangement of the discrete system, by considering that the right hand side Neumann boundary condition \eqref{eq:neumann_condition_continuous} leads to  $u_{M+1}^n=u_{M-1}^n$, for any $n$. Therefore,   for any $n=0,1,\dots,N-1$ 
 we consider the following complete numerical scheme
\begin{equation}\label{eq:heat_scheme}
\begin{split}
	u_{m}^{n+1}                      & =  \bar{\Delta}\, u_{m+1}^n + (1- 2\bar{\Delta}) u_m^n +  \bar{\Delta}\,u_{m-1}^n, \quad m=1,\dots,M-1;  \\
 u_{0}^{n+1} &=\widetilde{\psi}^{n+1}, \quad 
 u_{M}^{n+1} =(1- 2\bar{\Delta}) u_M^n + 2 \bar{\Delta}\,u_{M-1}^n; \\
  u_0^0  &= \widetilde{\psi}^0 \quad u_{m}^{0}                      =  0, \quad m=1,\dots,M.
\end{split}
\end{equation} 
The discrete bounded process $\left(\widetilde{\psi}^n\right)_n$ in \eqref{eq:psi_tilde_discrete} is given by the transform \eqref{eq:stochastic_boundary_condition_s_c} of the process solution of the LSST scheme \eqref{eq:LSTT_scheme}.
System \eqref{eq:heat_scheme} is equivalent to solve  for any $n=0,\dots, N-1$ the following  linear system for the internal nodes
\begin{equation}\label{eq:heat_matrix_scheme}
    U^{n+1} = A U^n +   \bar{\Delta}\, \widetilde{U}_0^{n},
\end{equation}
where  $U^n=(u^n_1,...,u^n_{M})^T \in \mathbb R^{M}$ for any $n=0,\dots, N-1$,  the matrix $A\in \mathbb R^{M}\times \mathbb R^{M}$ is given by
\begin{equation*}%\label{eq:def_matrix_A}
    A = 
    \begin{bmatrix}
         1- 2\bar{\Delta} &    \bar\Delta   &   0                           & 0     & 0 & \dots & 0\\
         \bar{\Delta}       & 1- 2\bar{ \bar{\Delta}}  &     \bar{\Delta}   & 0    & 0 & \dots & 0\\
       0                               &   \bar{\Delta} & 1-2  \bar{\Delta}      &  \bar{\Delta}  & 0&  \dots & 0\\
        \vdots                               & \vdots & \vdots     &\vdots  & \vdots&  \vdots & \vdots\\
        0     & \ldots                          &0 & 0&   \bar{\Delta}  &1-2   \bar{\Delta} & \bar{\Delta}\\
0& \ldots& 0 &0&0& 2 \bar{\Delta} & 1- 2 \bar{\Delta}.
    \end{bmatrix} 
\end{equation*}
Equation \eqref{eq:heat_matrix_scheme} is associated to the   boundary condition $\widetilde{U}_0^n$ at   $x=0$, coupled with the second rescaling in  \eqref{eq:stochastic_boundary_condition_s_c}, and the $t=0$ initial condition $U^0$
\begin{equation}\label{eq:boundary_initial_condition_U} 
    \begin{split}
        \widetilde{U}_0^n&= \left( \widetilde{\psi}^n,0,\ldots,0,0\right)^T;\\
        U^0&=(0,\ldots,0)^T.
    \end{split}
\end{equation}  
By an iteration of \eqref{eq:heat_matrix_scheme}, we get that a solution of \eqref{eq:heat_matrix_scheme}- \eqref{eq:boundary_initial_condition_U} is, for any $n\in\{1,...,N\}$
 \begin{equation}\label{eq:heat_matrix_scheme_solution}
    U^{n} = A^n U^0 +  \bar{\Delta} \sum_{j=0}^{n} A^{n-j}\, \widetilde{U}_0^{j}.
\end{equation}
Mean-square stability of the solution to a random partial differential equation refers to the property that the expected value of the square of the solution remains limited in the amount of growth that can occur in a finite time. We first have a result of pathwise stability.
Let us consider the Euclidean norm in space for $U^n$ and $\widetilde{U}_0$, respectively
\begin{equation}
   \label{def:euclidean_discrete} 
   \|U^n \|_{\Delta_x} = \left(\Delta_x \sum_{m=0}^{M}\left|u_m^n\right|^2\right)^{1/2}, \quad   \|\widetilde{U}_0 \|_{N,\Delta_t} = \left(\Delta_t \sum_{j=0}^{N}\left|\widetilde{U}_0^j\right|^2\right)^{1/2},
\end{equation}

\medskip

A stability notion, \cite{2004_Strikwerda}, for the initial-boundary value problem \eqref{eq:heat_matrix_scheme} is provided in the following.
\begin{definition}
   The discrete scheme \eqref{eq:heat_matrix_scheme} is $L^2$-pathwise stable if, for any $n$ such that $0\le n\Delta_t \le T$, and for any $\omega\in \Omega$, there exists a positive constant $C_T$ such that
   \begin{equation*}%\label{def:stability_U}
   \|U^n (\omega)\|_{\Delta_x}\le C_T \left( \|U^0 \|_{\Delta_x} + \|\widetilde{U}_0  (\omega)\|_{n,\Delta_t}\right).
   \end{equation*} Clearly, if the scheme is $L^2$-pathwise stable, it is also -$L^2$stable in mean. Indeed, following the classical stability results in the deterministic case, the following results may be stated.
\end{definition}
\begin{proposition}
\label{prop:stability_u}
Suppose that assumption \eqref{eq:condition_SDE_parameters} on the parameters of the stochastic boundary condition holds. Then for \begin{equation}
    \bar{\Delta} \le 1/2    \label{eq:first_stability_condition_Delta}
\end{equation}  the random process, solution of \eqref{eq:heat_matrix_scheme}- \eqref{eq:boundary_initial_condition_U} is pathwise stable; precisely, there exists a constant $C$ such that, for any $n\in \{1,\ldots, N\}$
\begin{equation*}%\label{eq:prop_stability_U_Deltax}
\|U^{n}(\omega) \|_{\Delta_x}   \le C_{T,\Delta_x} \widetilde\eta, 
\end{equation*} where $\widetilde\eta$   is defined in \eqref{eq:stochastic_boundary_condition_s_c_limit}.
Moreover, the mean square stability holds. Furthermore, the stability property holds also in the maximum norm case
both pathwise and in mean 
\begin{equation}\label{eq:prop_stability_U_supnorm}
\mathbb E\left[\|U^{n}  \|_{\max}\right]   \le C_{T} \widetilde\eta.
\end{equation}
\end{proposition}
\begin{proof} 

Let us observe that the matrix $A$ is similar to the matrix symmetric
$$
   \widetilde{A} = 
    \begin{bmatrix}
         1- 2\bar{\Delta} &    \Delta   &   0                           & 0     & 0 & \dots & 0\\
         \bar{\Delta}       & 1- 2\bar{ \bar {\Delta}}  &     \bar{\Delta}   & 0    & 0 & \dots & 0\\
       0                               &   \bar{\Delta} & 1-2  \bar{\Delta}      &  \bar{\Delta}  & 0&  \dots & 0\\
        \vdots                               & \vdots & \vdots     &\vdots  & \vdots&  \vdots & \vdots\\
        0     & \ldots                          &0 & 0&   \bar{\Delta}  &1-2   \bar{\Delta} & \sqrt{2} \bar{\Delta}\\
0& \ldots& 0 &0&0& \sqrt{2} \bar{\Delta} & 1- 2 \bar{\Delta}
    \end{bmatrix} ,
$$ i.e. $\widetilde{A}=S^{-1}A S,$ with $S=diag(1,1,..,1,\sqrt{2})\in \mathbb R^M.$
 Then 
$$
 \|A^n\|_2 =  \left\| \left(S^{-1}\widetilde{A} S\right)^n\right\|_2=\left\| S^{-1}\widetilde{A}^n S \right\|_2\le  \| S^{-1} \|  \| \widetilde{A}^n \| \| S\|_2\le \frac{1}{\sqrt{2}} \rho\left(\widetilde{A}\right)^n,  $$
where $\rho\left(\widetilde{A}\right)$ is the largest eigenvalue of the symmetric matrix $\widetilde{A}$ which is of the form
  $|1-4\bar{\Delta}\sin^2\left(\xi\right)|$, with $\xi\in(0,1),$ a function of $\Delta_x$, \cite{1995_thomas}. %Tomas Example3.1.10 pag117. 
  Since 
$ |1-4\bar{\Delta}\sin^2\left(\xi\right)|\le 1$ if and only if $\bar{\Delta}\le 1/2.$ 
Hence, by \eqref{eq:first_stability_condition_Delta} we get that, for any $n\in\{1,...,N\}$
$$\|A^n\|_2 \le 1.$$
Finally, by \eqref{eq:boundary_initial_condition_U} , \eqref{eq:heat_matrix_scheme_solution}, and \eqref{def:euclidean_discrete}  we get
\begin{equation*}
    \begin{split}
   \|U^{n+1}(\omega) \|_{\Delta_x} &\le  
   \| U^0(\omega) \|_{\Delta x} + \frac{\Delta_t}{\Delta_x^2}  \left\|  \sum_{j=0}^{n}  \widetilde{U}_0^{j}\right\|_{\Delta_x}\le  \frac{\Delta_t}{\Delta_x^2}\left(\Delta_x   \sum_{j=0}^{n}\left|\psi^n\right|^2\right)^{1/2}
   \\& \le  \frac{\Delta_t^{1/2}}{\Delta_x^{1+1/2}}\left(\Delta_t  \sum_{j=0}^{n}\left|\psi^n\right|^2\right)^{1/2}  = \frac{{\bar{\Delta}}^{1/2}}{\Delta_x^{1/2}}  \| \widetilde{\Psi} \|_{n,\Delta_t}\le C_{T,\Delta_x} \widetilde\eta,
    \end{split}
\end{equation*}
that is the thesis is proven. The mean case and \eqref{eq:prop_stability_U_supnorm} may be easily proven.    
\end{proof}
 
\subsection{The random non linear dynamics with deterministic boundary condition.}
Given the  solution  $\{u^n_m\}_{m,n}$ of the discrete heat equation \eqref{eq:heat_scheme} or equivalently, \eqref{eq:heat_matrix_scheme} -\eqref{eq:boundary_initial_condition_U} %with $C^\beta, \beta<1/2$ 
and boundary condition in $x=0$ given by $\{\widetilde\psi^n\}_{n}$,  we take the discrete version of  equation \eqref{eq:v} for $v$, coupled with the explicit solution for the calcite $c$ given by \eqref{eq:c_formula}, where $s$ is defined by \eqref{eq:def_splitted_s}. For any $m\in\{1,\ldots,M\}, n\in \{1,\ldots,N\},$ let $s_m^n = u_m^n + v_m^n$ and $c^n_m$ be the discrete version of $s$ and $c$, respectively.   Again, by using  an explicit first order finite difference approximation for \eqref{eq:v} we obtain the following forward system of equation for any $n=0,\ldots,N-1$ and $m=1,...,M-1$
\begin{eqnarray*}
    v_m^{n+1}  &=& v_m^n +  \bar{\Delta} \left(v_{m+1}^n - 2v_m^n + v_{m-1}^n\right)\notag \\&
   & +  \bar{\Delta}  \frac{\varphi_2}{4\varphi(c_m^n)} \left(c_{m+1}^n-c_{m-1}^n\right)\left(v_{m+1}^n-v_{m-1}^n\right) \notag\\
        &&  +  \bar{\Delta} \frac{\varphi_2}{4\varphi(c_m^n)} (c_{m+1}^n-c_{m-1}^n)(u_{m+1}^n-u_{m-1}^n)\notag \\
        & & + \Delta_t \lambda c_m^n  \left(   \varphi_2(v_m^n+u_m^n) -    1)(v_m^n+  u_m^n \right); \\
         c_m^{n+1} &= &c_m^n \,e^{-  \lambda\Delta_t \,(v_m^n+u_m^n)\, \varphi(c_m^n)},
\end{eqnarray*}
coupled with the discrete solution $u^n_m$ given by  \eqref{eq:heat_scheme}.
By considering the functions  
\begin{equation*}  %  \label{def:b^n(c)_h^n}
   \begin{split} b^n_m= b^n_m(c)&:=\frac{\varphi_2}{4\varphi(c_m^n)}(c_{m+1}^n-c_{m-1}^n)=\frac{\varphi{(c^n_{m+1})}-\varphi{(c^n_{m-1})}}{4\varphi(c_m^n)};\\
    h_m^n=h_m^n(c,u,v) &:= \lambda   \Delta_t  c_m^n \left( \varphi_2(v_m^n+u_m^n) -1\right),
    \end{split}
\end{equation*}
a small rearrangement leads to the following scheme for $(v_m^n,c_m^n)$
\begin{eqnarray}
 \label{eq:scheme_v_c_1}
       v_m^{n+1}  &=&  \bar{\Delta}  \left(1+b^n_m(c) \right)  v_{m+1}^n  +  \bar{\Delta} \left(1-b^n_m(c)\right) v_{m-1}^n   \\
    & &+  \left[1-2 \bar{\Delta} +   h_m^n(c,u,v) \right ]v_m^n \notag\\  
    & &+  \bar{\Delta} b^n_m(c) u_{m+1}^n + h_m^n(c,u,v) u_m^n-  \bar{\Delta} \,  b^n_m(c) u_{m-1}^n;\notag \\
    c_m^{n+1} &=& c_m^n \,e^{-  \lambda\Delta_t \,(v_m^n+u_m^n)\, \varphi(c_m^n)}, 
    \label{eq:scheme_v_c_2} 
    \end{eqnarray}
coupled with $u^n_m$ in \eqref{eq:heat_scheme}
and initial condition given by
\begin{equation}\label{eq:initial_condition_scheme_v_c} 
	 v_m^0 = \,\bar{s}_0\le \widetilde\eta, \quad    c_m^{0}=\bar{c}_0,  \quad  m\in \{1,\ldots,M-1\}. \\
\end{equation}
The boundary condition at the endpoint reads as $
   v_{M+1}^{n}= v_{M-1}^{n},$ so that, for any $n=1,\ldots, N$, $c_{M+1}^{n} = c_{M-1}^{n}$ and $ 
     b_M^n := \frac{B}{4\varphi(c_M^n)}(c_{M+1}^n-c_{M-1}^n)=0$. Therefore, this leads to the boundary equation for $v$, that is, for any $n=1,\ldots,N-1$
\begin{equation}    \label{eq:boundary_condition_scheme_v_c} 
    \begin{split}
      v_0^n&= 0;\\
      v_M^{n+1}&= 2\bar\Delta v_{M-1}^n + (1-2\bar\Delta + h_M^n)v_M^n +  h_M^n u_M^n.  
    \end{split}     
\end{equation}
We may express the system \eqref{eq:scheme_v_c_1}, \eqref{eq:initial_condition_scheme_v_c} and \eqref{eq:boundary_condition_scheme_v_c} in a matrix form. For any $n=0,\dots, N-1$, let $V^n$ be the vector $V^n=(v^n_1,...,v^n_{M})^T \in \mathbb R^{M},$   and let the matrices $G(n),P(n)\in \mathbb R^{M}\times \mathbb R^{M}$ be the following
\begin{align*}
 G(n) & = 
   \begin{bmatrix}
        1-2\bar{\Delta} + h_1^n & \bar{\Delta}(1 + b_1^n ) & 0 & \dots &  0 \\
        \bar{\Delta}(1-b_2^n)   & 1-2\bar{\Delta} + h_2^n  & \bar{\Delta}(1 + b_2^n ) & \dots & 0 \\
        \vdots & \vdots & \vdots & \vdots & \vdots \\
        0 &  \dots & \bar\Delta(1-b_{M-1}^n) & 1-2\bar\Delta + h_{M-1}^n & \bar\Delta(1+b_{M-1}^n) \\
        0   & \dots  & 0 & 2\bar{\Delta} & 1-2\bar{\Delta} + h_{M}^n
    \end{bmatrix};
    \\
   P(n) & = 
    \begin{bmatrix}
        h_1^n & \bar{\Delta} b_1^n  & 0 & \dots & \dots & 0 \\
        -\bar{\Delta}b_2^n   & h_2^n  & \bar{\Delta} b_2^n & \dots & \dots & 0 \\
        \vdots & \vdots & \vdots & \vdots & \vdots & \vdots\\
        0 & \dots & 0 & -\bar\Delta b_{M-1}^n & h_{M-1}^n & \bar\Delta b_{M-1}^n\\
        0  & \dots &\dots &  0   &  0   & h_{M}^n
    \end{bmatrix}.
\end{align*}
Then, for $n=0,1,\ldots N$, the system \eqref{eq:scheme_v_c_1}, \eqref{eq:initial_condition_scheme_v_c} and \eqref{eq:boundary_condition_scheme_v_c} reads as
\begin{align}\label{eq:v_system_matrix_form}
    V^{n+1} & = G^n\, V^n + P^n\, U^n + \bar\Delta \widetilde{{V}}^n,
\end{align}
where $\widetilde{{V}}^n  = \left(- b_1^n \tilde\psi^n, 0, \ldots, 0 \right)$.

\subsubsection{Boundedness and stability}

The first step of the analysis regards  boundedness result for  $(s^n_m, c_m^n) = (u^n_m+v^n_m, c_m^n)$, solution of the following system, easily obtained 
      from systems \eqref{eq:heat_scheme} and  \eqref{eq:scheme_v_c_1}-\eqref{eq:scheme_v_c_2}; for any 
      $ (n,m) \in \{ 0,\ldots, N-1 \}\times \{1,\ldots, M\}$, 
   \begin{eqnarray} 
  s_m^{n+1}  &= & \bar{\Delta}  \left(1+\frac{\varphi{(c^n_{m+1})}-\varphi{(c^n_{m-1})}}{4\varphi(c_m^n)}\right)  s_{m+1}^n +  \bar{\Delta} \left(1-\frac{\varphi{(c^n_{m+1})}-\varphi{(c^n_{m-1})}}{4\varphi(c_m^n)}\right) s_{m-1}^n \notag\\  \label{eq:scheme_s_c_1}
    && +  \left[1-2 \bar{\Delta} -  \lambda \Delta_t    c_m^n \left(1-\varphi_2 s^n_m\right)\right ] s_m^n;\\\label{eq:scheme_s_c_2}
    c_m^{n+1} &= &c_m^n \,e^{-  \lambda\Delta_t s_m^n\, \varphi(c_m^n)}.
\end{eqnarray}
  \begin{proposition}\label{prop:boundeness_scheme_s_c}
     Let us suppose that, given  the initial condition \eqref{eq:initial_condition_scheme_v_c} and the first stability condition  \eqref{eq:first_stability_condition_Delta},  the following further conditions are satisfied
     \begin{eqnarray}\label{eq:positive_condition_s_1}
        % 5\varphi(\bar{c}_0) &>&\varphi_1
        \bar{c}_0 &<& - \frac{4}{5}\frac{\varphi_1}{\varphi_2}= \frac{4}{5}\frac{\varphi_1}{\left|\varphi_2\right|};\\
\label{eq:positive_condition_s_2} \Delta_t&\le& \frac{\Delta_x^2}{  2  +  \lambda     \bar{c}_0 \Delta_x^2 (1-   \varphi_2  \widetilde\eta )  }.
     \end{eqnarray}
     Then,  the solution of the system \eqref{eq:scheme_s_c_1}-\eqref{eq:scheme_s_c_2} is such that, for any 
      $ (n,m) \in \{ 1,\ldots, N \}\times \{1,\ldots, M\}$, and any $\omega\in\Omega$
      $$(s^n_m(\omega), c_m^n(\omega)) \in \mathbb [0,\widetilde{\eta}) \times   [0,\bar{c}_0].$$
  \end{proposition}
\begin{proof}
    Let us observe that from the expression of $c^{n+1}_m$, it is clear that if $s^n_m\in [0,\tilde{\eta}),$ then $c^{n+1}_m \in[0,c^n_m]$. This is true for $n = 0$. Let us suppose it is true for all $k \le n$. Then  for any $k\in\{1,\ldots,n+1\},$  $c^{k}_m \in[0,\bar{c}_0]$ and then, for any $m\in\{1,\ldots,M\},$ 
\begin{equation*}%\label{eq:prop_positivity_porosity}
0<\varphi(\bar{c}_0)=\varphi_1+\varphi_2 \bar{c}_0\le \varphi(c_m^{k} ) \le \varphi_1< 1.
\end{equation*}
From condition \eqref{eq:positive_condition_s_1} one may prove that   the second coefficient of \eqref{eq:scheme_s_c_1} is positive; indeed, \begin{eqnarray*}
   1-\frac{\varphi{(c^n_{m+1})}-\varphi{(c^n_{m-1})}}{4\varphi(c_m^n)}
     &=&   \frac{4\varphi(c_m^n)-\varphi{(c^n_{m+1})}+\varphi{(c^n_{m-1})}}{4\varphi(c_m^n)} \\
     &\ge&\frac{5 \varphi(\bar{c}_0) -\varphi(c^n_{m+1})}{4\varphi(c_m^n)} \\
      &\ge& \frac{5 \varphi(\bar{c}_0) -\varphi_1   }{4\varphi(c_m^n)} =\frac{4 \varphi_1+5\varphi_2 \bar{c}_0     }{4\varphi(c_m^n)} > 0.
\end{eqnarray*}

In a similar way,
 \begin{eqnarray*}
   1+\frac{\varphi{(c^n_{m+1})}-\varphi{(c^n_{m-1})}}{4\varphi(c_m^n)}
     &\ge&   \frac{5\varphi(\bar{c}_0)-\varphi{(c^n_{m-1})}}{4\varphi(c_m^n)}  \ge   \frac{5 \varphi(\bar{c}_0)-\varphi_1 }{4\varphi(c_m^n)}  \ge 0.
\end{eqnarray*}
Then,  for the induction hypothesis we may prove the upper bound for $s^{n+1}_m$; since $(1-\varphi_2 s^n_m)>0$, we get \begin{eqnarray*} 
  s_m^{n+1}  &\le & \bar{\Delta}    s_{m+1}^n +  \bar{\Delta}   s_{m-1}^n +  \left[1-2 \bar{\Delta} -  \lambda \Delta_t    c_m^n \left(1-\varphi_2 s^n_m\right)\right ] s_m^n \\
  &\le & 2 \bar{\Delta}   \widetilde\eta  +  \left[1-2 \bar{\Delta}   \right ] s_m^n \\
  &\le & 2 \bar{\Delta}   \widetilde\eta  +  \left[1-2 \bar{\Delta}   \right ] \widetilde\eta = \widetilde\eta
\end{eqnarray*}
For the proof of the lower bound, we need to prove that also the last coefficient of the equation is positive;  we have the following 
\begin{eqnarray*}
     1-2 \bar{\Delta} -  \lambda \Delta_t    c_m^n \left(1-\varphi_2 s_m\right) &=&1-2 \frac{\Delta_t}{\Delta_x^2}-  \lambda \Delta_t    c_m^n  + \lambda \Delta_t    c_m^n \varphi_2 s_m \\
     &\ge&1-2 \frac{\Delta_t}{\Delta_x^2}-  \lambda \Delta_t    \bar{c}_0  + \lambda \Delta_t \varphi_2 \bar{c}_0   \widetilde\eta\\
     &=&\frac{1}{\Delta_x^2}\left[\Delta_x^2-\Delta_t\left(2  + \Delta_x^2\lambda     \bar{c}_0  (1-   \varphi_2   \widetilde\eta)\right)\right]
\end{eqnarray*}
The last is greater than zero if and only if condition \eqref{eq:positive_condition_s_2} is satisfied. Then, the thesis is achieved. 
\end{proof}

\begin{corollary}
Let us suppose that the hypotheses of Proposition \ref{prop:boundeness_scheme_s_c} are satisfied. Then the solution of \eqref{eq:heat_scheme} and  \eqref{eq:scheme_v_c_1}-\eqref{eq:scheme_v_c_2} are such that for any 
$ (n,m) \in \{ 1,\ldots, N \}\times \{1,\ldots, M\}$, and any $\omega\in \Omega,$
\begin{equation}\label{eq:v_negative}
-\widetilde\eta \le -u_m^n(\omega)\le   v^n_m(\omega) \le \widetilde\eta -u_m^n.
\end{equation}
\end{corollary}
\begin{proof}
      The result easily follows from  Proposition \ref{prop:boundeness_scheme_s_c} and the fact that $u_M^n \in [0,\eta)$.\\      
\end{proof}
 \begin{proposition}
        Let us suppose that the hypotheses of Proposition \eqref{prop:boundeness_scheme_s_c} are satisfied. Then there exist positive constants $C_{\Delta_x,T,\bar{x}}, C_{\Delta_x,T}$  for any 
      $ n\in \{ 1,\ldots, N \} $, the solution $V^n$ of \eqref{eq:v_system_matrix_form} is such that
      \begin{equation}\label{eq:stability_pathwise_L2_v}
          \|V^{n}(\omega) \|_{\Delta_x} \le   C_{\Delta_x,T,\bar{x}} \widetilde \eta, 
      \end{equation} 
      and \begin{equation*}%\label{eq:stability_pathwise_L2_v_max}
\|V^{n}(\omega) \|_{\max}   \le C_{T} \widetilde\eta.
\end{equation*}
The same stability results in mean follow.
 \end{proposition}
 \begin{proof}
      From the bound \eqref{eq:v_negative}, we get that  for any 
      $ (n,m) \in \{ 1,\ldots, N \}\times \{1,\ldots, M\}$, 
       $$
      \sum_{m=1}^M  |v_m^n|^2\le    \sum_{m=1}^M|u_m^n |^2+  M | \widetilde\eta|^2.  $$
      Then, for 
      $V^n=\left(v_1^n,\ldots,v_M^n\right)^T, H=id(\widetilde\eta)\in \mathbb R^M$ the following bound, for any $n=0,\ldots,N-1,$ $$  \|V^{n+1}(\omega) \|_{\Delta_x} \le \|U^{n+1}(\omega) \|_{\Delta_x} + \|H \|_{\Delta_x} \le   C_{\Delta_x,T } + \bar{x}\widetilde \eta\le C_{\Delta_x,.T,\bar{x}} \widetilde \eta
        $$
Then, bound \eqref{eq:stability_pathwise_L2_v} follows from Proposition \ref{prop:stability_u}. The rest of the thesis easily follows with the same arguments and by considering the mean.
\end{proof}
\begin{remark}  Condition \eqref{eq:positive_condition_s_1} is equivalent to a condition on the initial porosity, that is $\varphi(\bar{c}_0) >\varphi_1/5$. This means that the initial porosity needs to be not too small. On the other hand, condition \eqref{eq:positive_condition_s_2} says that $\bar\Delta\le 1/2$ is not sufficient for both the positivity and stability of $v$, but a strong restriction is needed. Moreover, condition \eqref{eq:positive_condition_s_2} involves the velocity of the reaction $\lambda$: thus, the stability of solutions is ensured even in the accelerated regime.
\end{remark}
The study of the convergence rate of the scheme is beyond the purposes of this work; however,  we numerically estimate the pathwise spatial accuracy order. We look 
at the ratios of the distances of two solutions computed for halved successive $\Delta_x$. For a fixed time step $\Delta t$, the \emph{spatial pathwise accuracy order} of a discrete function $g$ is estimated, for any $\omega \in \Omega$ as
$$
p_g(\omega)= \log_2
\left(\frac{\left\|g^N_{\cdot} (\Delta_x,\omega) - g^N_{\cdot} (\Delta_x/2,\omega)\right\| }{\left\|g^N_{\cdot} (\Delta_x/2,\omega) - g^N_{\cdot} (\Delta_x/4,\omega)\right\|} \right),
$$
 where  $\|\cdot \|$ is the  Euclidean norm in space \eqref{def:euclidean_discrete} and $g_m^n(\delta,\omega)$ is the discretized function $g$ at $m$-th position at $n$.th time, in the case of a spatial mesh size $\delta$ and for the realization $\omega$.

\begin{table}[!h]
    \centering
     \begin{tabular}{|c|ccc|ccc|}
    \hline
   $\Delta_x$  	&\multicolumn{3}{|c|}{ {$\|\rho_{\Delta_x}-\rho_{\Delta_x/2}  \|_2$}}		&\multicolumn{3}{|c|}{	$p_\rho$} 	\\
   \hline
	&$\omega_1$&	$\omega_2$&	$\omega_3$&	$\omega_1$	&$\omega_2$	&$\omega_3$  \\
    \hline
0.125&	0.016&	0.014&	0.017&	1.304&	1.316&	1.283\\
0.0625&	0.006	&0.006	&0.007&	1.157&	1.167&	1.146\\
0.03125&0.003  &0.002&	0.003&&	&\\		
0.015625&	&&	&&&\\									
    \hline
    \hline
   $\Delta_x$  	&  \multicolumn{3}{|c|}{			$\|c_{\Delta_x}-c_{\Delta_x/2}  \|_2$		}	&\multicolumn{3}{|c|}{$p_c$	}	\\
   \hline
   &$\omega_1$&	$\omega_2$&	$\omega_3$&	$\omega_1$	&$\omega_2$	&$\omega_3$  \\
   \hline
   0.125&	0.101	&0.095	&0.104	&1.330	&1.337	&1.329 \\
0.0625&	0.040&	0.037&	0.041&	1.181&	1.185&	1.180 \\
0.03125& 0.018	&0.016&	0.018&&	&\\		
0.015625&	&&	&&&\\
\hline
    \end{tabular}
    \caption{Numerical estimation of the Euclidean $L^2$ norms and the spatial accuracy order for three different random paths at the boundary for $\rho$ and $c$ at final time $T=1$.  Parameters:  $\sigma_3=1,\varphi_1=0.2$ and $\lambda=1$; other parameters in Table \ref{table:parameters}.}
    \label{tab:accuracy_order} 
\end{table}
 Let the time step $\Delta_t = 1.907e-06$ be fixed. We consider $\Delta_x=0.125$ and the successive halved $\{\Delta_x/2, \Delta_x/4,\Delta_x/8\}$ and let $(\rho^N_m, c^N_m)_{m}$ be the discrete solution at the time $T=t_N$. 
Tables \ref{tab:accuracy_order} show the results for a specific set of parameters, even though different parameters do not show very different results. We consider only four mesh sizes and we see how results are rather stable with respect to different random boundary paths, with an estimation of the order one accuracy.

\section{Numerical Sampling: pathwise results}\label{se:numerical_sampling}

We sample the solution of the system by considering a time horizon $[0,T]=[0,1.5]$ and a spatial domain $[0,\bar{x}]=[0,1.5].$ In the following, we show the numerical results of the random    PDE  in the smaller space window $[0,1],$ to avoid the influence at the boundary condition at $\bar{x}=1.5.$ 
The time step is
$\Delta_t= 1.99 e^{-5},$ while the spatial one  is $ \Delta_x= 10^{-2}.$

\medskip

We aim to test the behaviour of the diffusion-advection-reaction system under the influence of the stochastic dynamical boundary condition. Concerning the sampling of the complete system, we perform simulations for different values of diffusion parameters for the SDE;  furthermore, to catch the role of the activation energy in the penetration of the sulphur dioxide and the consequent degradation of the calcium carbonate, we also consider different reaction rates $\lambda$, by considering a slow ($\lambda=1$) and fast ($\lambda_3=100$) regimes. The case of a diffusion coefficient $\sigma_0=0$ in equation \eqref{eq:SDE_for_SO2} corresponds to the deterministic case; in particular, we consider a configuration in which the deterministic boundary condition for $\rho$ in $x=0$ is not constant as in   \cite{2004_ADN,2007_GN_CPAA}, but it is a function $\Psi_0\in C_b([0,T])$, i.e. for $t\in [0,T]$
\begin{equation}
\label{eq:deterministic_boundary_rho}
\rho_{\sigma_0}(t,0)=\Psi_0(t)=\gamma\left( 1-e^{-\alpha t}\right).
\end{equation}
Note that the constant boundary condition $\rho_{\sigma_0}(t,0)=1$ studied in   \cite{2004_ADN} corresponds to the time asymptotic limit of equation \eqref{eq:deterministic_boundary_rho}. 

 \medskip

Table \ref{table:parameters} shows the parameters considered in the numerical sampling.

\begin{table}[h!]
\begin{center}
\begin{tabular}{|l||c|c|c|c|c|c|c|c| }
\hline
& & & & & &   &  &  \\[-1em]
SDE&\,\,	$\alpha$\,\, & \,\,$\gamma$\,\, & \,\,$\eta$ \,\,& \,\,$\sigma_0$ \,\, &\,\,$\sigma_1$ \,\, &   \,\,$\sigma_2$ \,\, & \,\,$\sigma_3$ \,\, &  \\
& & & & & & & & \\[-1em]
\hline
& & & & & & & & \\[-1em]
 &7	 & 	1   &   1.5  &  0   &  0.25  &0.7  &  1&\\
[2pt]
 \hline
 \hline
& & & & & && & \\[-1em]
 PDE  & \,\, $\bar{s}_0$\,\,& \,\, $\bar{c}_0$\,\,&\,\, $\varphi_1$\,\,&\,\,$\widetilde{\varphi}_1$ \,\,&\,\, $\varphi_2$\,\,&\,\,  $\lambda_1$\,\,&  $\lambda_2$\,\,&\,\, $\lambda_3$ \\
 & & & & & & && \\[-1em]
 \hline
 & & & & & & & &\\[-1em]
  & 0& 10& 0.2 &0.7& -0.01&1& 10& 100 \\[-1em]
  & & & & & && \\
 \hline
\end{tabular}	
\caption{Parameter sets for the simulation of the PDE system \eqref{eq:heat_scheme} and \eqref{eq:scheme_v_c_1}-\eqref{eq:boundary_condition_scheme_v_c}.}\label{table:parameters}
\end{center}\end{table}

\subsection{Random heat equation and the different contributions to the sulphation: slow dynamics}

First of all, we consider the case of slower dynamics, where the activation energy or reaction rate $\lambda=\lambda_1=1$. Furthermore, a first case of a porosity function \eqref{eq:def_porosity} with $\varphi_1=0.2$  is examined. 
In this section we illustrate the dynamics pathwise, focusing on a single realization of the involved random processes. Following the splitting strategy, we firstly discuss the paths of $\left(u^n,\widetilde{\psi}^n\right)_n$  as the solution of the discrete stochastic system
\eqref{eq:heat_scheme}, together with $(v^n,c^n)_n$, or, equivalently, $(s^n=u^n+v^n,c^n)_n$ via system \eqref{eq:scheme_v_c_1}-\eqref{eq:boundary_condition_scheme_v_c}. 

 Figure \ref{fig:plot_u_different_sigma}  shows a sample path of the random heat equation \eqref{eq:heat_scheme} coupled with conditions \eqref{eq:boundary_initial_condition_U}  for different diffusion coefficients $\sigma$ for the dynamical stochastic boundary condition in $x=0$.  Randomness is concentrated at the boundary: thanks to the smoothing properties of the Laplace operator, the solution becomes more regular whenever it moves inside the domain. As the noise intensity $\sigma$ grows, the corresponding
solution oscillates more and more and oscillations spread all over the domain.  We may also appreciate the contribution of the porosity acting at the boundary condition, as expressed by \eqref{eq:psi_tilde_discrete}: when the SDE reaches a stationary distribution, the action of increasing porosity let the boundary  profile decrease. This is much more evident in the deterministic case (top-left), and for the smaller $\sigma$ (top-right).

\begin{figure}[h!]
    \centering
    \begin{subfigure}{0.39\linewidth}
        \centering
        \includegraphics[width=\linewidth]{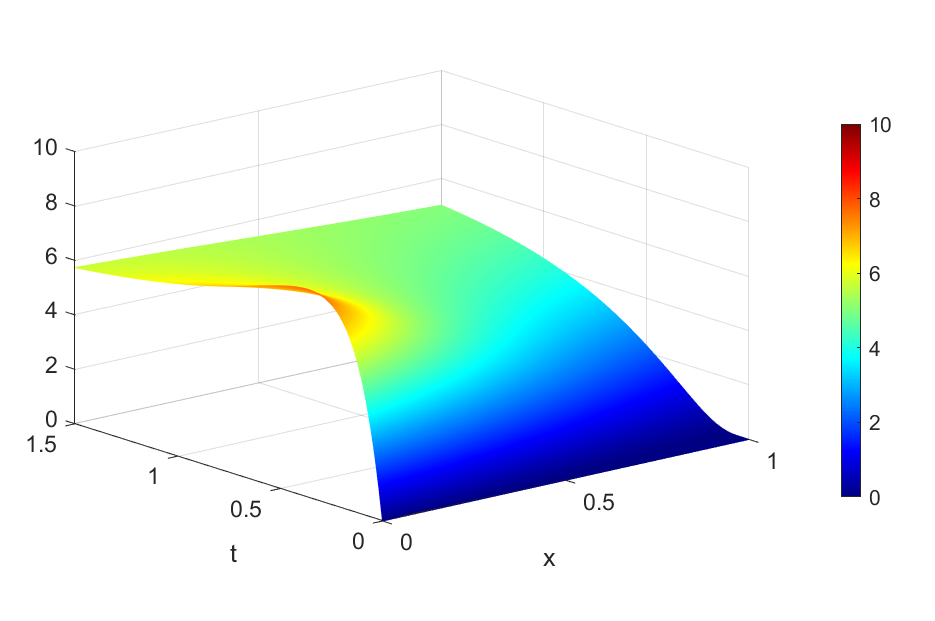}
    \end{subfigure}
    \begin{subfigure}{0.39\linewidth}
        \centering
        \includegraphics[width=\linewidth]{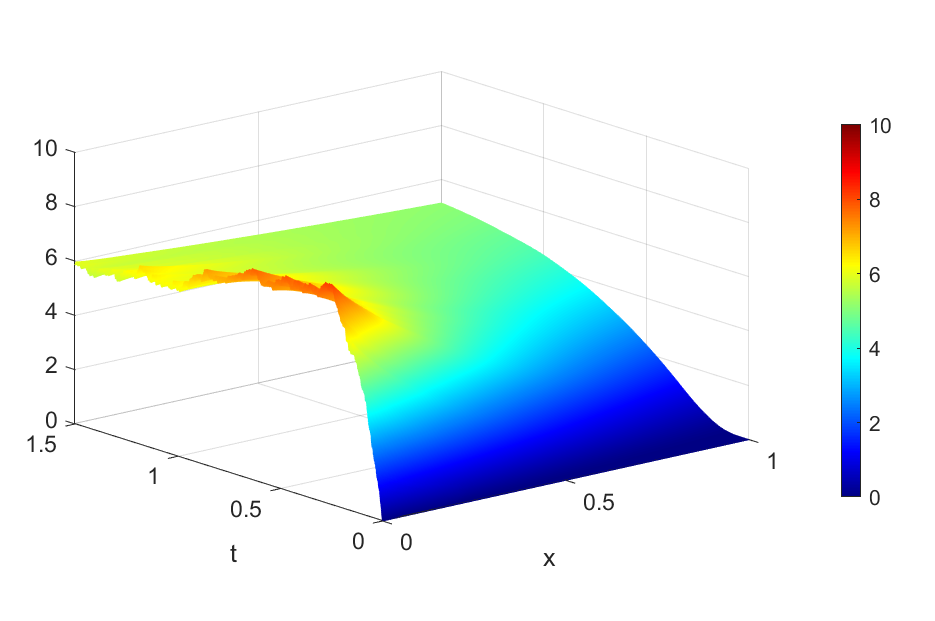}
    \end{subfigure}
     
    \begin{subfigure}{0.39\linewidth}
        \centering
        \includegraphics[width=\linewidth]{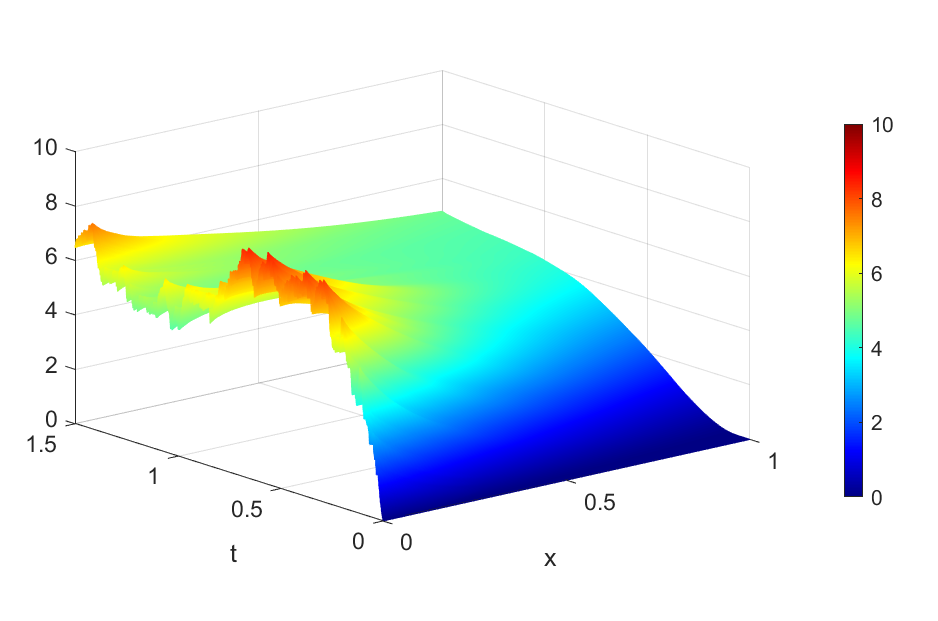}
    \end{subfigure}
    \begin{subfigure}{0.39\linewidth}
        \centering
        \includegraphics[width=\linewidth]{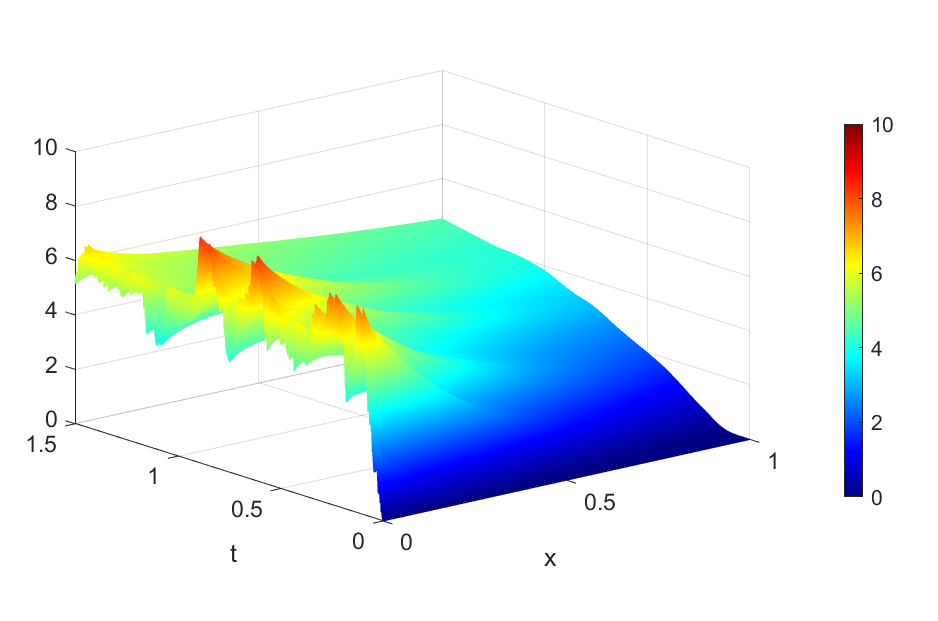}
    \end{subfigure}
   \caption{Random heat solution  in the case $\varphi_1=0.2$, for different values of $\sigma$. Top-left: $\sigma_0=0$; top-right: $\sigma_1=0.25$; bottom-left : $\sigma_2=0.7$;  bottom-right : $\sigma_3=1$. }\label{fig:plot_u_different_sigma}
\end{figure}

\vspace{-2cm}
\begin{figure}[h!]
    \centering
    \begin{subfigure}{0.39\linewidth}
        \centering
        \includegraphics[width=\linewidth]{Figure/Set_S2_P1_u_rescaled.eps}
    \end{subfigure}
    \begin{subfigure}{0.39\linewidth}
        \centering
        \includegraphics[width=\linewidth]{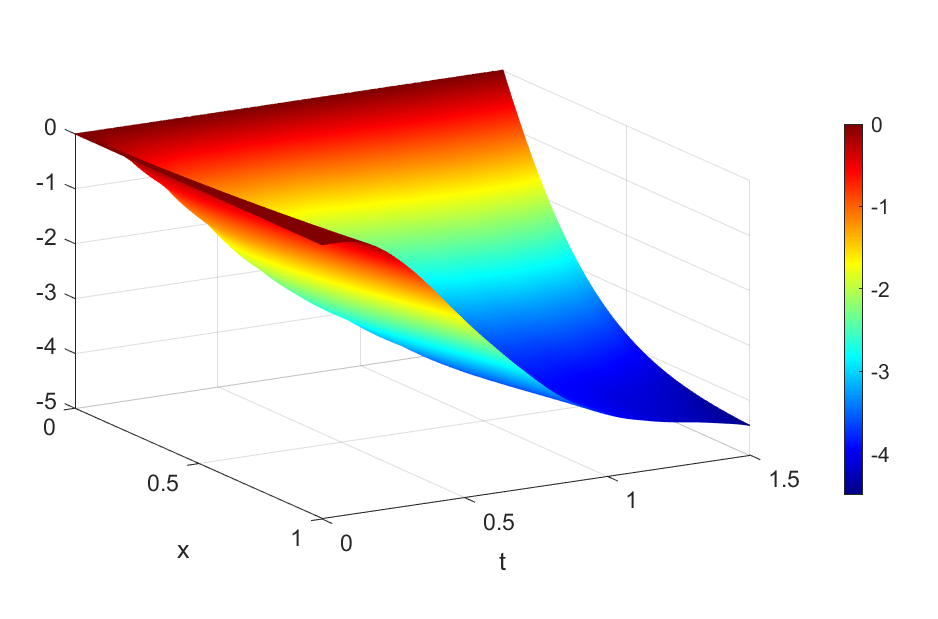}
    \end{subfigure}
   
    \vspace{-0.1cm}
    \begin{subfigure}{0.39\linewidth}
        \centering
        \includegraphics[width=\linewidth]{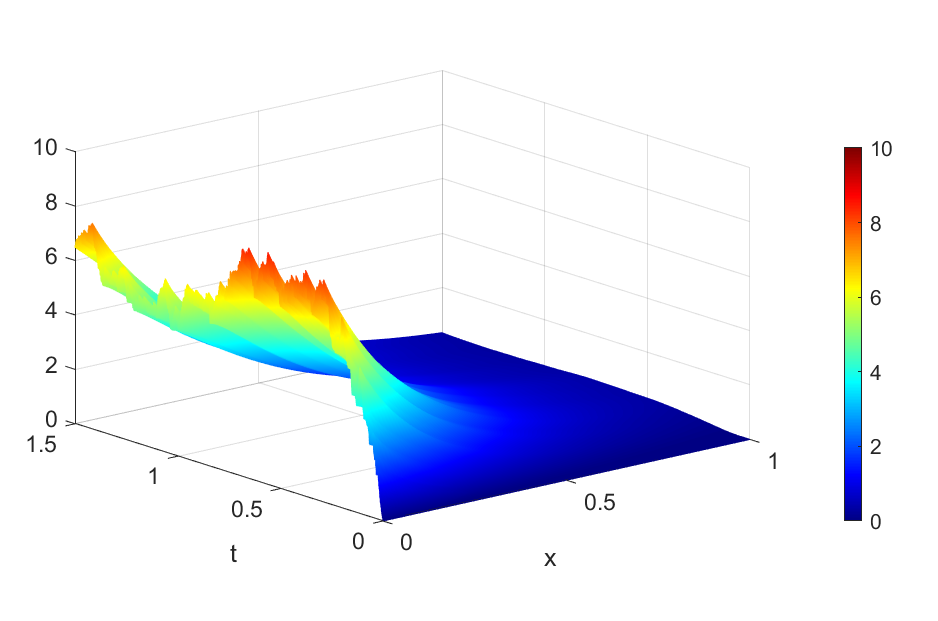}
    \end{subfigure}
    \begin{subfigure}{0.39\linewidth}
        \centering
        \includegraphics[width=\linewidth]{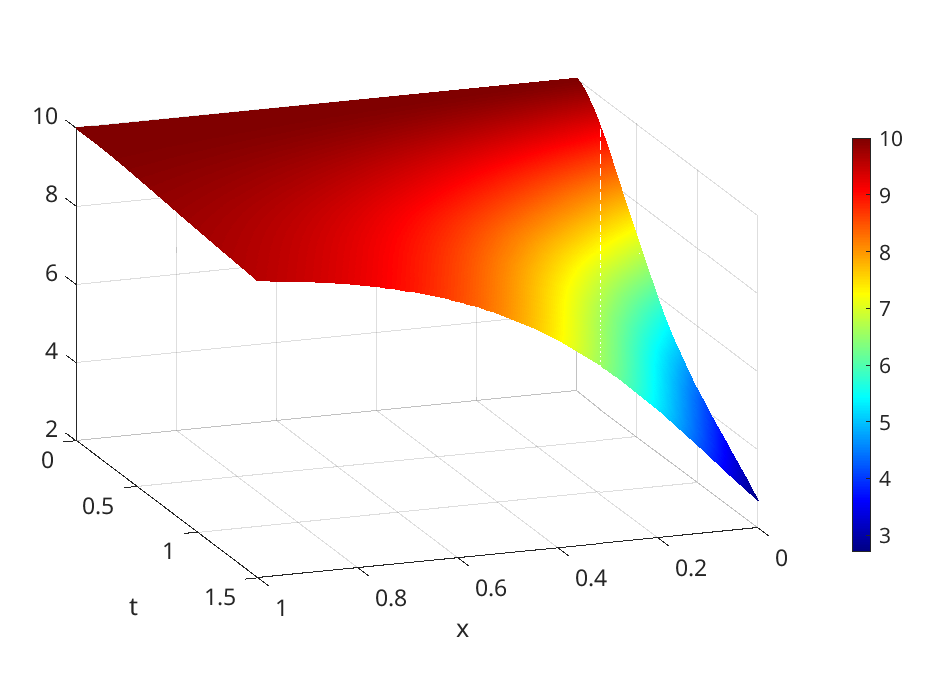}
    \end{subfigure}
  \caption{ Pathwise different contributions to the complete PDE with $\sigma_2=0.7$. First row:   random heat equation solution $u$  (left) and   solution $v$ of the non linear equation   \eqref{eq:scheme_v_c_1} (right). Second row: evolution of sulphur dioxide   $s=u+v$ (left) and calcite    $c$ in \eqref{eq:scheme_v_c_2}. }	
   \label{fig:plot_usvc}
\end{figure}

 \newpage
Figure \ref{fig:plot_usvc} depicts the single contribution of a path of the process $(s=u+v,c)$ that may be derived by the splitting strategy. The case of diffusion coefficient at the boundary  $\sigma_{2}=0.7$ is considered. 
The solution $v$ of the nonlinear 
equation \eqref{eq:scheme_v_c_1} clearly includes the reaction of the sulphur dioxide; since the initial and boundary conditions are zero, from 
\eqref{eq:stochastic_boundary_condition_s_c_limit} and \eqref{eq:v_negative} it gives a complete negative contribution bounded from below by $-\widetilde\eta= -15$. This is the reason why the sulphur dioxide porous concentration $s=u+v$ oscillates as $u$, but as soon as it moves away from the boundary, oscillations slow down to zero. As the profile of the calcite $c$ in \eqref{eq:scheme_v_c_2} concerns, we see how the degradation occurs as time increases, with a slower degradation of the calcite profile far from the boundary.  
Figure \ref{fig:plot_rho_different_sigma} illustrates the evolution of a pathwise solution of  $\rho= \varphi(c)s$ sampled, for any $(m,n)\in \{0,\dots,M\}\times\{0,\ldots,N\}$, by $$
\rho_m^n=(v_m^n+u_m^n)\, \varphi(c_m^n),
$$
with $ c_m^n$ in \eqref{eq:scheme_v_c_2}. As the noise at the boundary increases, the perturbation is much more detected far from the boundary, even though the effect is mitigated by the low reaction rate and the quite low porosity of the transformed material.
 
\begin{figure}[h!]
 \centering
    \begin{subfigure}{0.4\linewidth}
        \centering
        \includegraphics[width=\linewidth]{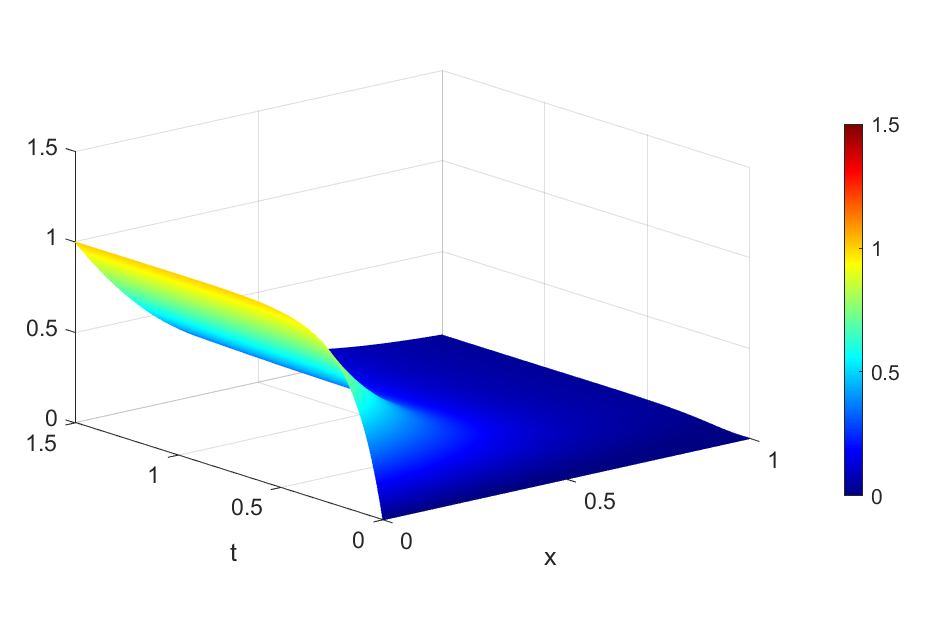}
    \end{subfigure}
    \begin{subfigure}{0.4\linewidth}
        \centering
        \includegraphics[width=\linewidth]{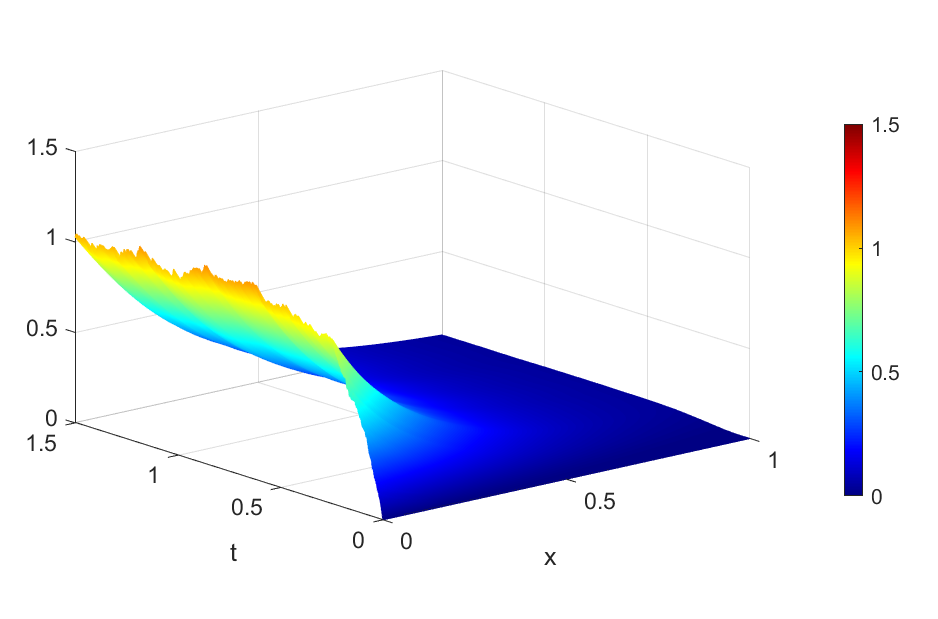}
    \end{subfigure}
  
    \vspace{-0.5cm}
    \begin{subfigure}{0.4\linewidth}
        \centering
        \includegraphics[width=\linewidth]{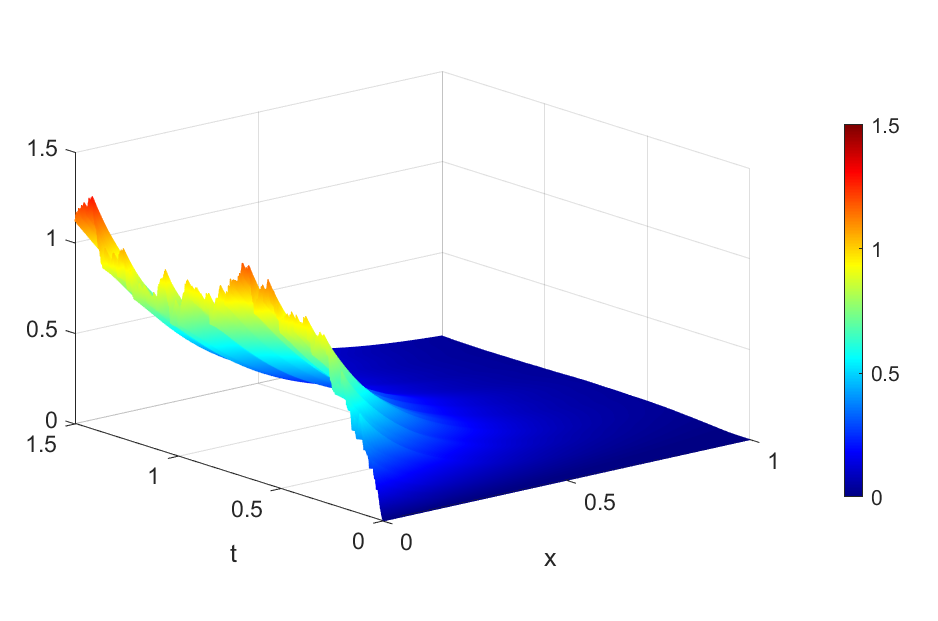}
    \end{subfigure}
    \begin{subfigure}{0.4\linewidth}
        \centering
        \includegraphics[width=\linewidth]{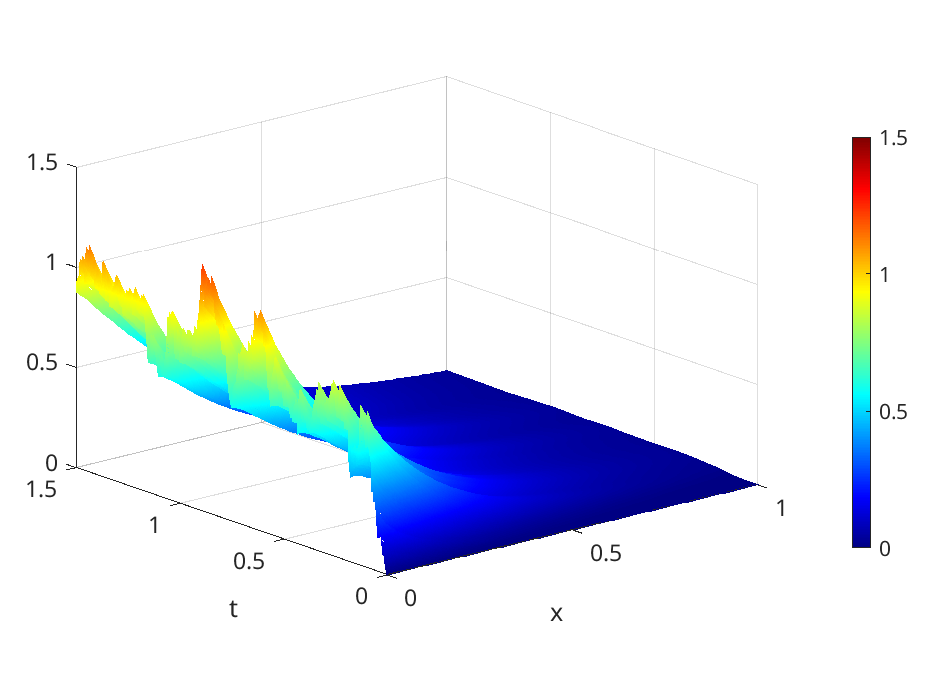}
    \end{subfigure}

\caption{The role of the diffusion $\sigma$. A single path of the \ch{SO2} concentration $\rho=\varphi s$ for different levels of boundary noise. Top-left: $\sigma_0=0$; top-right: $\sigma_1=0.25$; bottom-left: $\sigma_2=0.7$;  bottom-right: $\sigma_3=1$. }
\label{fig:plot_rho_different_sigma}
\end{figure}

\smallskip

Gypsum porosity may reach a level of 70\%. The higher the porosity, the faster the penetration of the sulphur dioxide and, thus, the degradation. The effect may be detected in Figure \ref{fig:confronto_lambda1_sigma07_phi1_02_07}: the evolutions of a realization of the \ch{SO2} and the corresponding calcium carbonate density for a lower porosity $\varphi_1=0.02$ are compared with the ones with higher porosity $\widetilde{\varphi}_1=0.7$. All other parameters involved remain unchanged.    

\begin{figure}[h!]
    \centering
    \begin{subfigure}{0.4\linewidth}
        \centering
        \includegraphics[width=\linewidth]{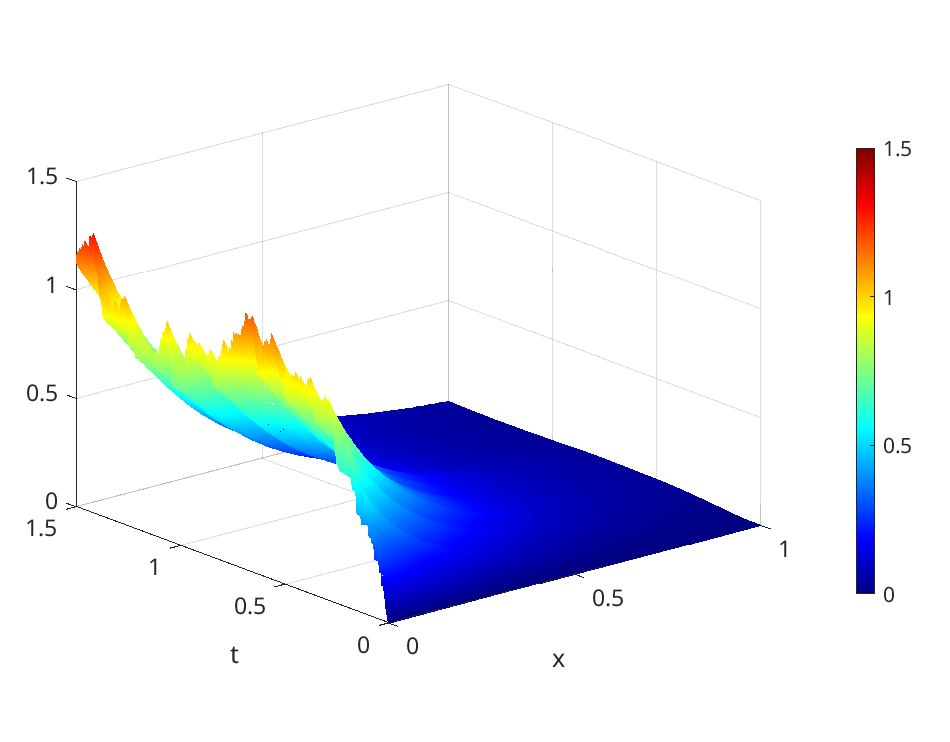}
    \end{subfigure}
    \begin{subfigure}{0.4\linewidth}
        \centering
        \includegraphics[width=\linewidth]{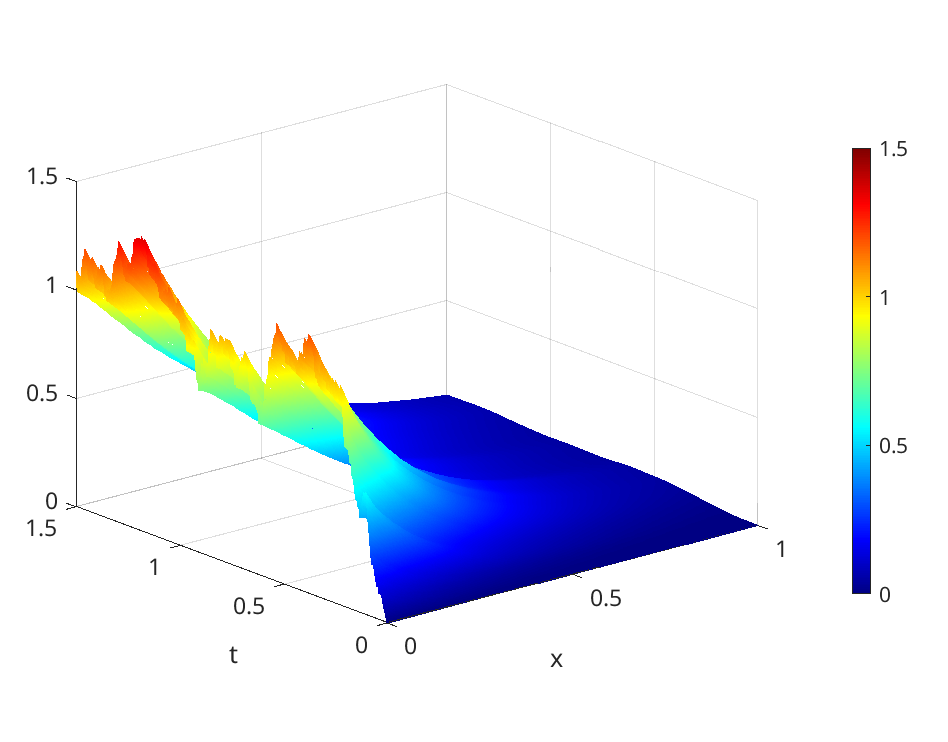}
    \end{subfigure}
    
    \vspace{-0.5cm}
    \begin{subfigure}{0.4\linewidth}
        \centering
        \includegraphics[width=\linewidth]{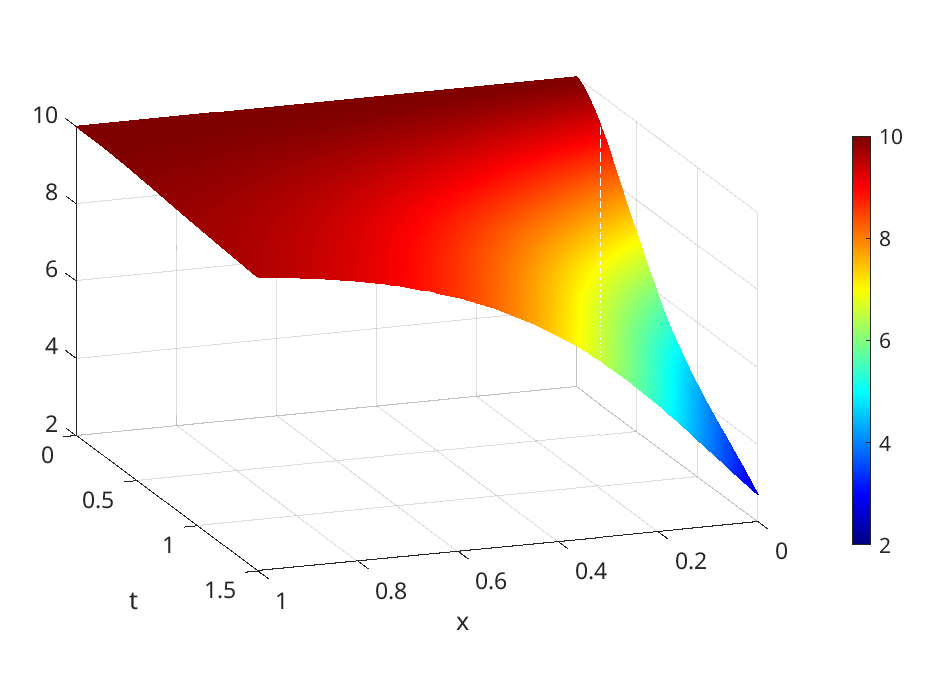}
    \end{subfigure}
    \begin{subfigure}{0.4\linewidth}
        \centering
        \includegraphics[width=\linewidth]{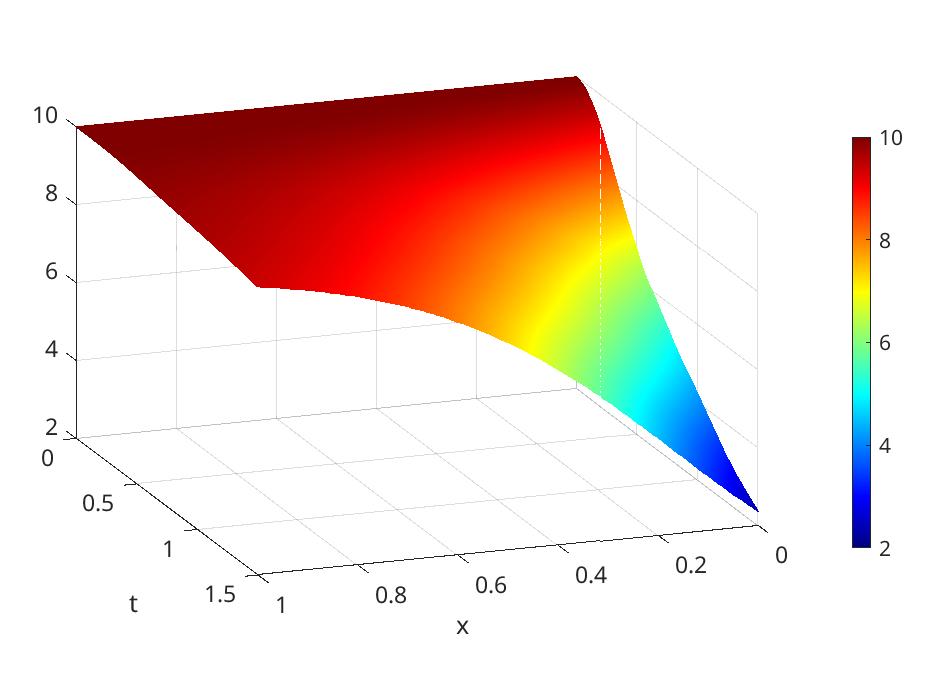}
    \end{subfigure}
   \caption{The role of the gypsum porosity via $\varphi_1$. Evolution of a realization of \ch{SO2} (first row) and the corresponding marble concentration (second row) for   $\varphi_1=0.02$  (left) and a higher porosity $\widetilde{\varphi}_1=0.7$ (right). Other parameters: $\lambda=1$,   $\sigma_2=0.7$. }
\label{fig:confronto_lambda1_sigma07_phi1_02_07}
\end{figure}

\subsection{Faster reactions: comparison of the sulphur dioxide and calcium  carbonate profiles}
Here, we consider different values of the parameter $\lambda\in \mathbb R_+$ in System \eqref{eq:PDE}, \eqref{eq:IC} and \eqref{eq:BC}.
As discussed in   \cite{2007_GN_CPDE},
an accelerated regime for the reaction by assuming different values $\lambda_i, \, i=2,3$ accordingly to Table \ref{table:parameters} is considered. 
This is equivalent to analyse the asymptotic behaviour of the system by accelerating time by a factor of $\lambda_i$ and taking a space expansion $\sqrt{\lambda_i} x$. By changing the parameter $\lambda$, the role of the activation energy in the penetration of the sulphur dioxide and the consequent degradation of the calcium carbonate is captured. 
Figure \ref{fig:confronto_lambda_sigma_1_phi1_02} shows the evolution of the couple $(\rho,c)$ for a larger time interval $ [0,5]$ and for increasing value of $\lambda$.  
Note that we reverse the angle view for $\rho$ to show both space and time evolution along the whole grid.
The solution has a very different qualitative aspect: as the interaction coefficient $\lambda$ increases, the transition zone gets smaller, leading to more evident calcite deterioration at the boundary of the sample, while the interior remains unaffected by sulphur dioxide. 
In particular in Figure \ref{fig:confronto_lambda_sigma_1_phi1_02} -(c), for $\lambda_3=100$, we may see the formation of a moving front. The study of the behaviour of the front distribution in the stochastic case from a theoretical point of view will be the subject of a future work.

\begin{figure}[h!]
    \centering
    \begin{subfigure}{0.4\textwidth}
        \centering
        \includegraphics[width=\textwidth]{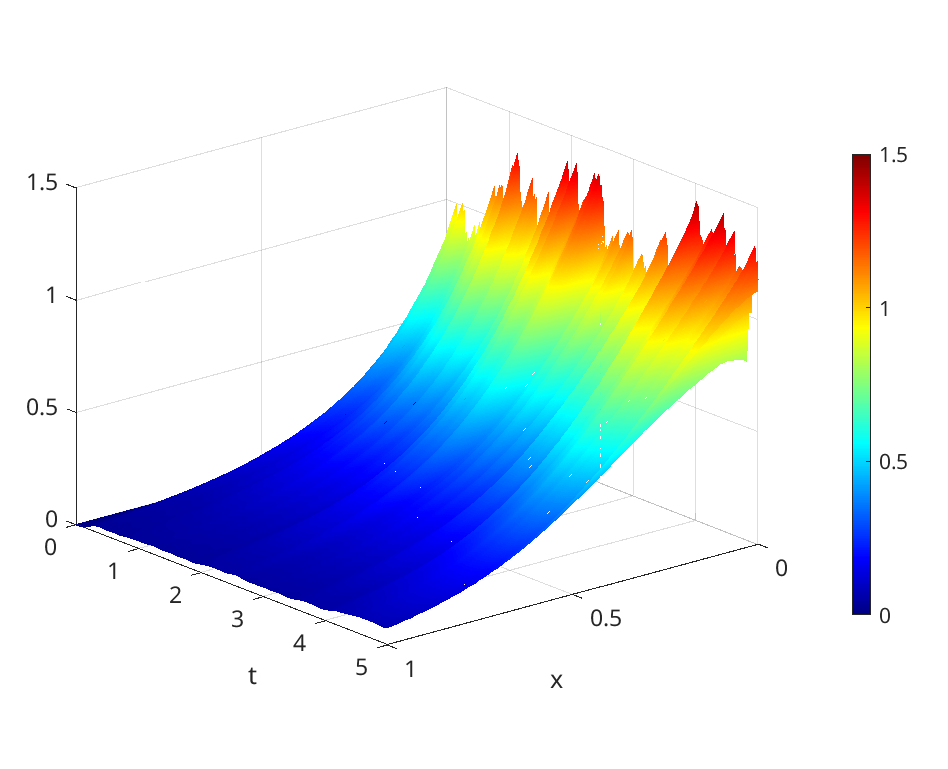}
    \end{subfigure}
         \begin{subfigure}{0.4\linewidth}
        \centering
        \includegraphics[width=\linewidth]{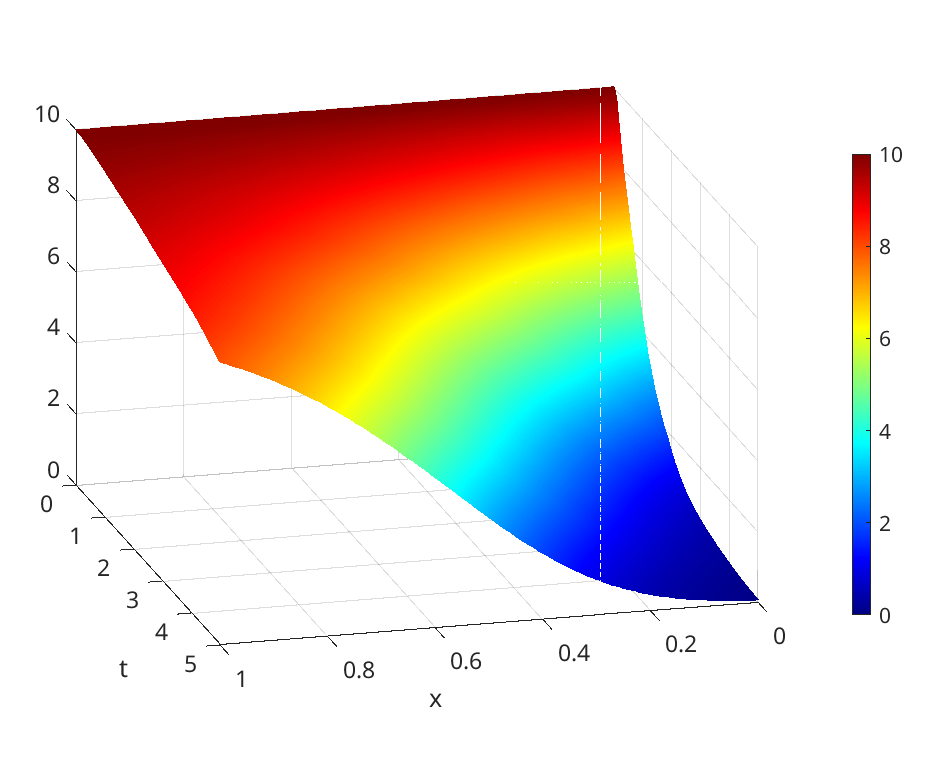}
    \end{subfigure}
    
    (a)
    
    \begin{subfigure}{0.4\textwidth}
        \centering
        \includegraphics[width=\textwidth]{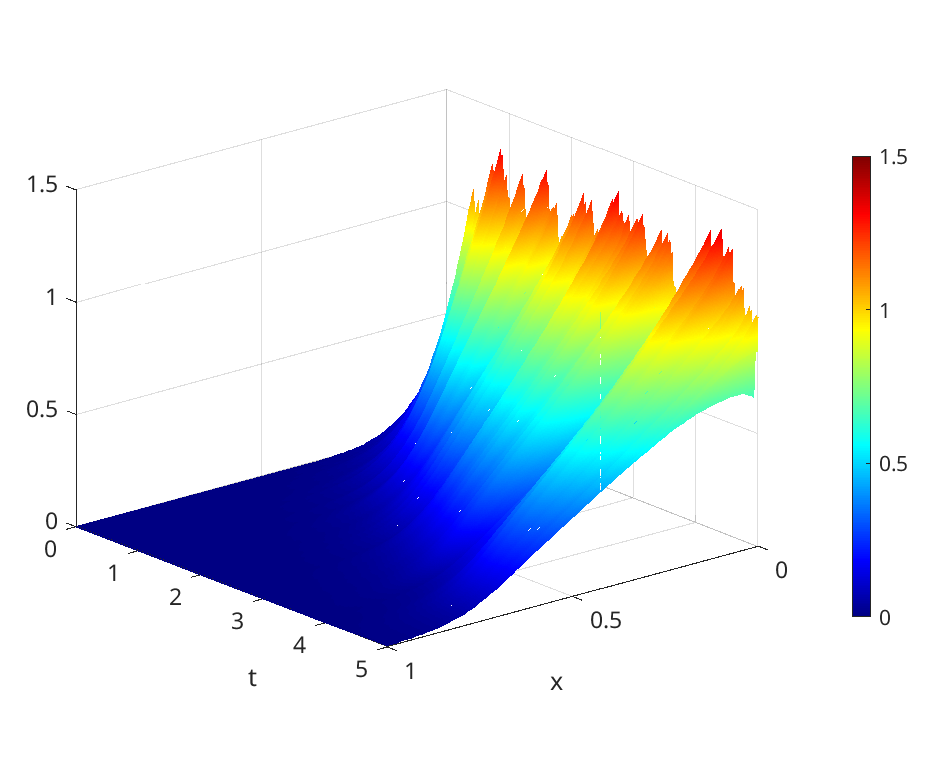}
    \end{subfigure}
    \begin{subfigure}{0.4\linewidth}
        \centering
        \includegraphics[width=\linewidth]{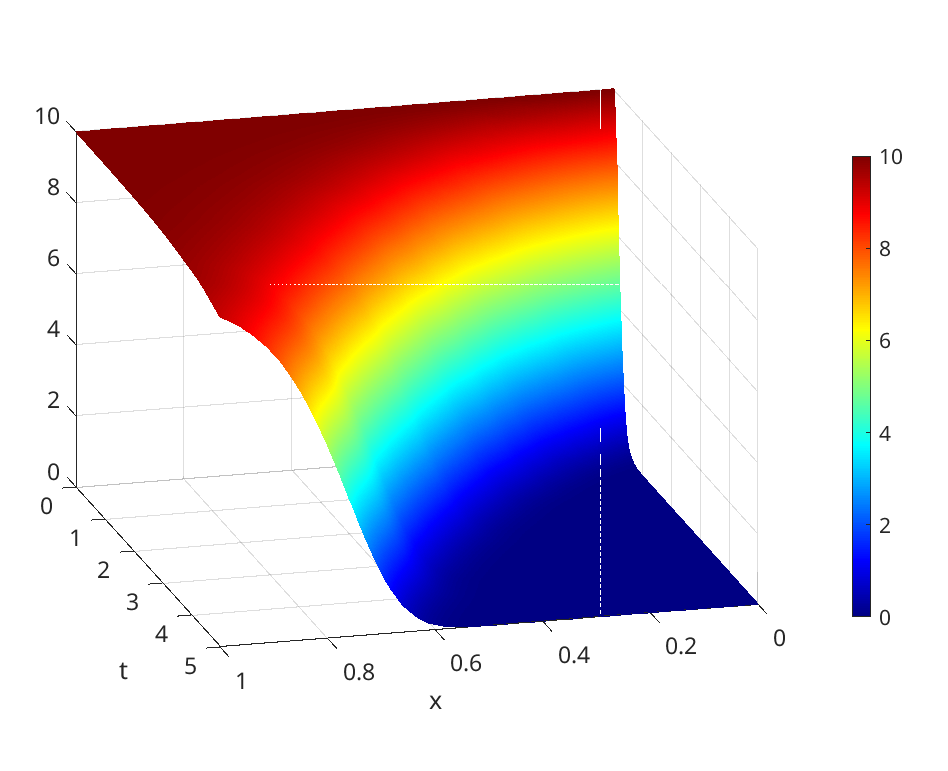}
    \end{subfigure}
    
    (b)
    
    \begin{subfigure}{0.4\textwidth}
        \centering
        \includegraphics[width=\textwidth]{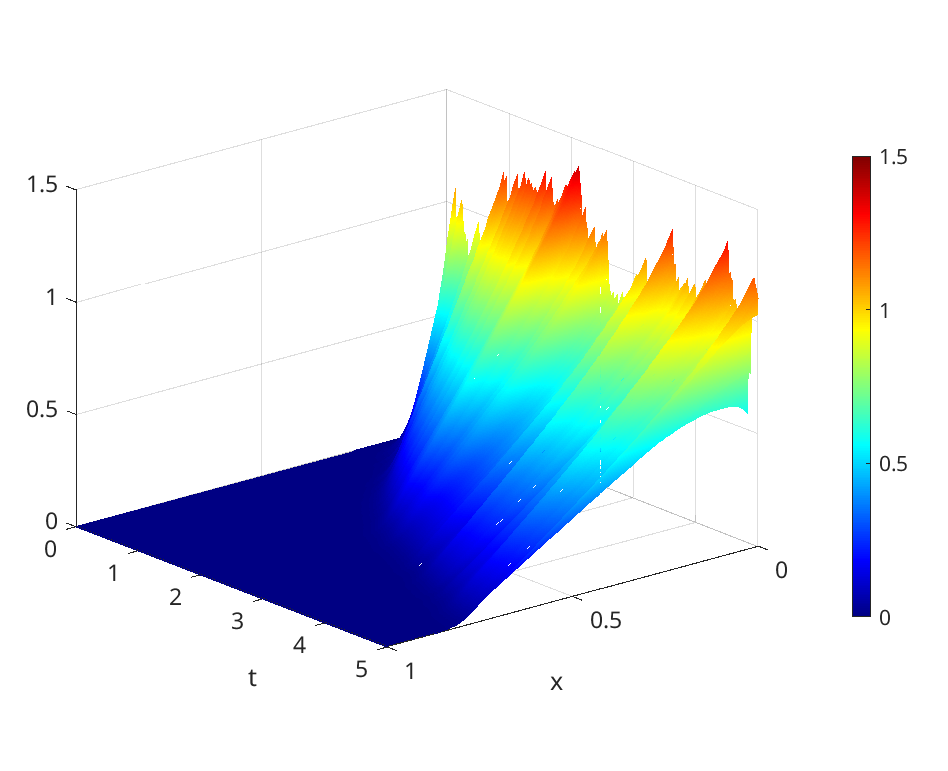}
    \end{subfigure}
        \begin{subfigure}{0.4\linewidth}
        \centering
        \includegraphics[width=\linewidth]{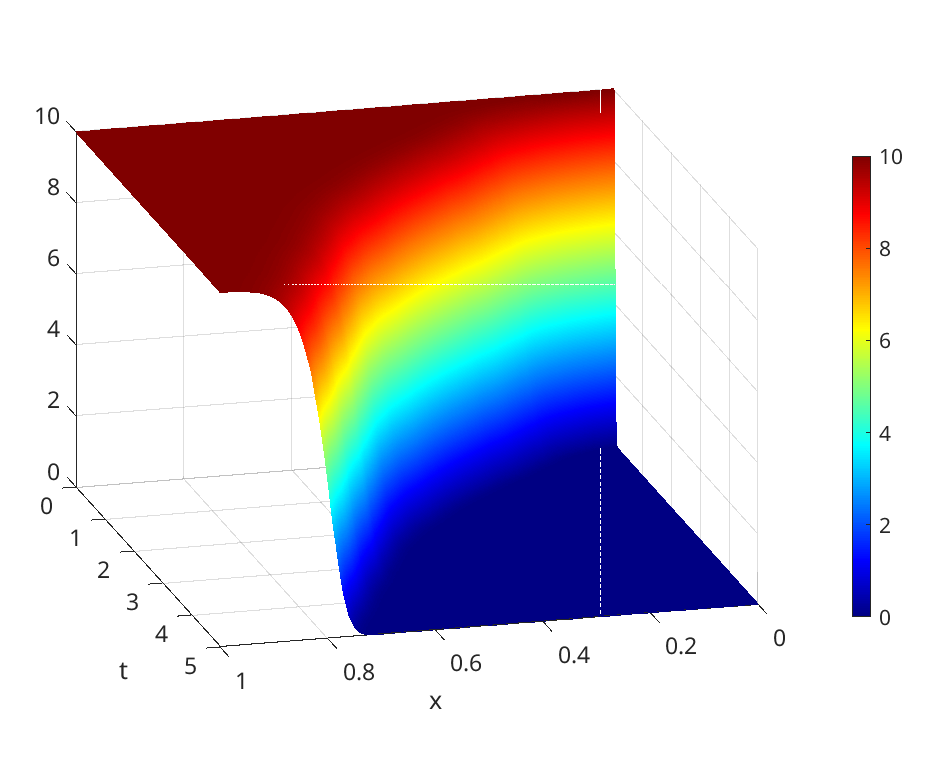}
    \end{subfigure}
    
    (c)
    
   \caption{The role of the activation energy $\lambda$. Evolution of a single realization of \ch{SO2} concentration (first column) and the corresponding marble density (second column) for increasing reaction rate: (a)  $\lambda_1=0.1$; (b) $\lambda_2=10$; (c)   $\lambda_3=100.$   Other parameters: $\varphi_1=0.2,\varphi_2=-0.01$, $\sigma_3=1$, $t\in [0,5]$. }
\label{fig:confronto_lambda_sigma_1_phi1_02}
\end{figure}

\section{Numerical Sampling:   estimation results}\label{se:numerical_sampling_estimation}  
 
Let us comment upon the fact that all the qualitative descriptions carried out in Section \ref{se:numerical_sampling} refer to a typical, but single realization of the random system. Now  we perform a statistical estimation of the dynamics for appreciating the variability of the process.  
We consider the estimation in $(t,x)\in [0,1.5]\times [0,1]$ by means of a sample  $\overline{N}=500$   trajectories of both processes $\rho$ and $c$. 
Concerning the variability of the distribution, Figure  \ref{fig:plot_quartile_rho_c} presents a graph of the 25\%, 50\% and 75\% percentile of the sample for each grid point. Hence,  the figure shows the intervals of the usual points of the distribution, that means the intervals within which the central 50\% of the distribution is located. We compare the case of  $\sigma_3=1$ as in Figure \ref{fig:confronto_lambda_sigma_1_phi1_02}, for both the slow reaction with $\lambda_1=1$ and the fast reaction case with $\lambda_3=100$.  
We observe that after a small time interval, all the percentiles at the boundary oscillate around some fixed points; in particular, the median oscillates around $\nu=1$. This suggests the existence of an invariant measure, coherently with the  invariant measure of the Pearson process at the boundary with unitary mean as in \eqref{eq:invariantmeasure}. Furthermore, Figure \ref{fig:plot_quartile_rho_c} shows how for the slow reaction case (a) close to $t=0$ the variability is  not detected and increases over time; furthermore fluctuations around the median are stronger closed to the boundary, even though they are present almost everywhere. A different scenario comes up in the case of fast reaction (b). As mentioned in Figure \ref{fig:confronto_lambda_sigma_1_phi1_02}, the transition zone gets smaller and it is located at the boundary; the interior is not touched by the sulphur dioxide penetration (left) and the variability rapidly goes to zero. The calcite deterioration (right) is very much significant  at the boundary: the formation of a front is visible and the variability is detectable, although it is very small.

 \begin{figure}[h!]
    \centering  
    \begin{subfigure}{0.43\linewidth}
        \centering
        \includegraphics[width=\linewidth]{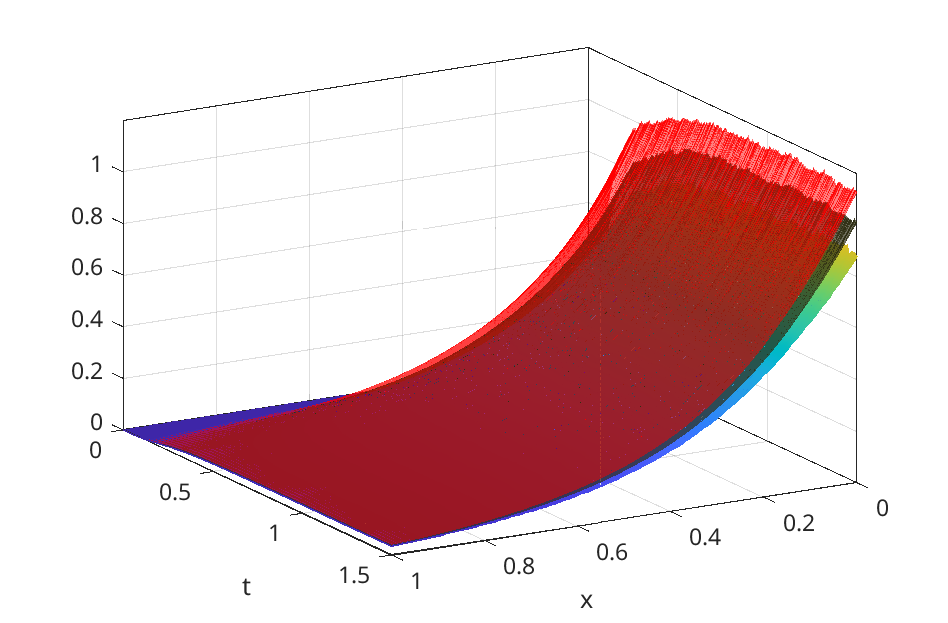}
    \end{subfigure}
    \begin{subfigure}{0.43\linewidth}
        \centering
        \includegraphics[width=\linewidth]{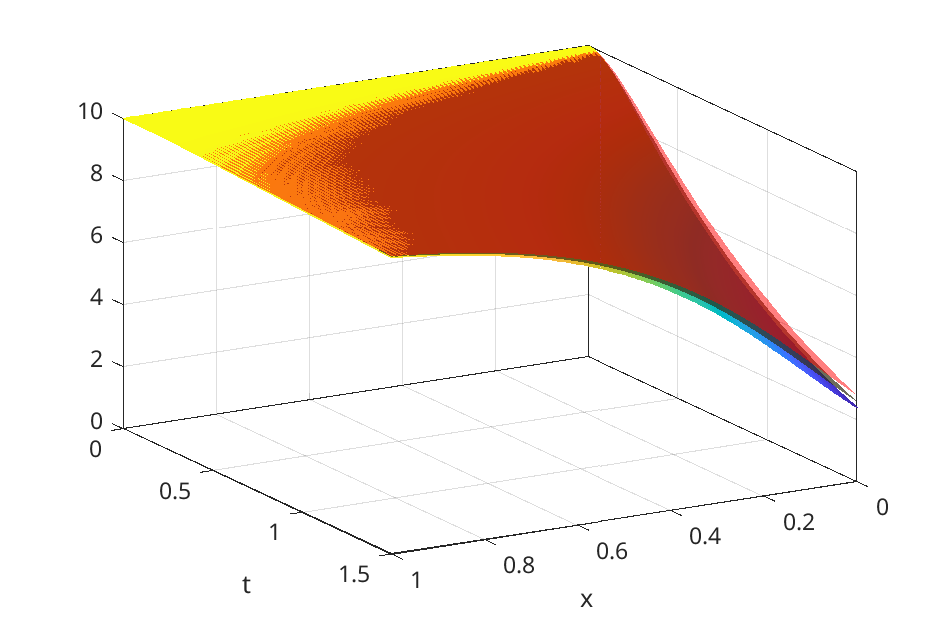}
    \end{subfigure}

    (a)

     \begin{subfigure}{0.43\linewidth}
        \centering
        \includegraphics[width=\linewidth]{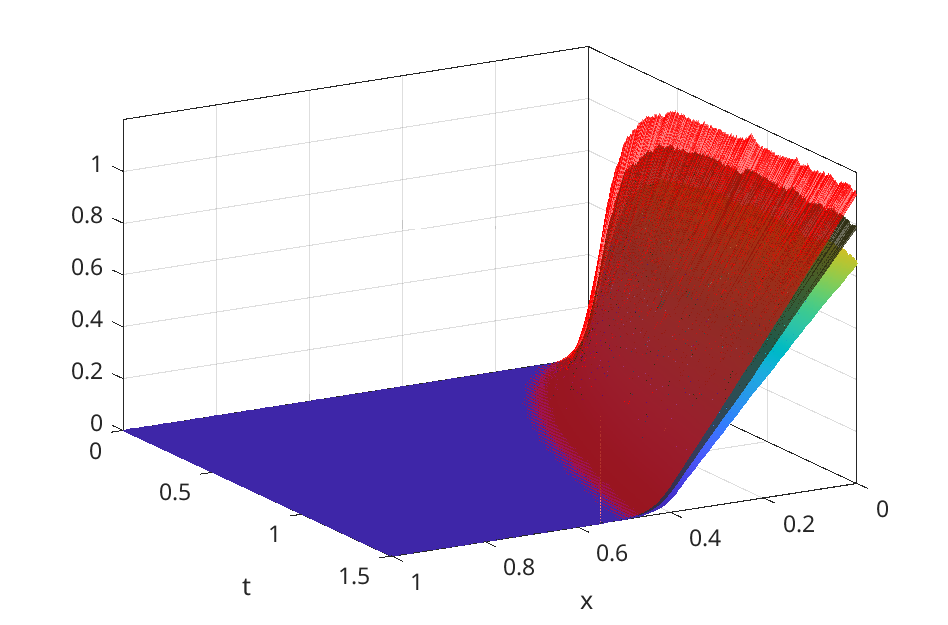}
    \end{subfigure}
    \begin{subfigure}{0.43\linewidth}
        \centering
        \includegraphics[width=\linewidth]{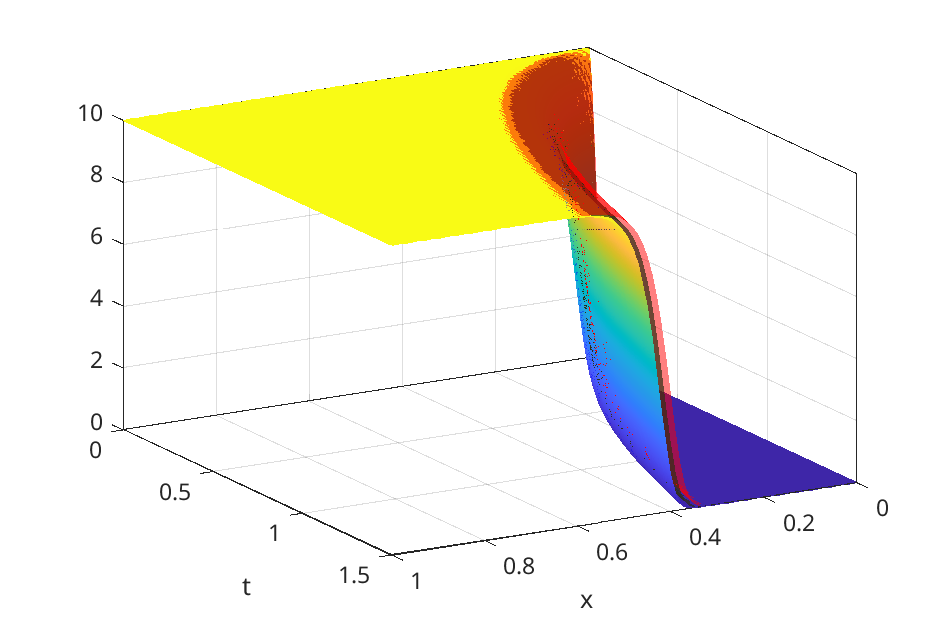} 
    \end{subfigure} 
      
    (b)

\caption{ Estimation of the distribution : 25\%   (green-blue), 50\% (black) and 75\% (red) percentile of an $\overline{N}=500$  sample of   the processes $\rho$ (left) and $c$ (right)  at each mesh point. Parameters: $\varphi_1=0.2$, $\sigma_3=1$. (a) slow reaction $\lambda_1=1$; (b) fast reaction $\lambda_3=100.$  }	\label{fig:plot_quartile_rho_c}
 \end{figure}
 \begin{figure}[h!]
    \centering  
    \begin{subfigure}{0.4\linewidth}
        \centering
        \includegraphics[width=\linewidth]{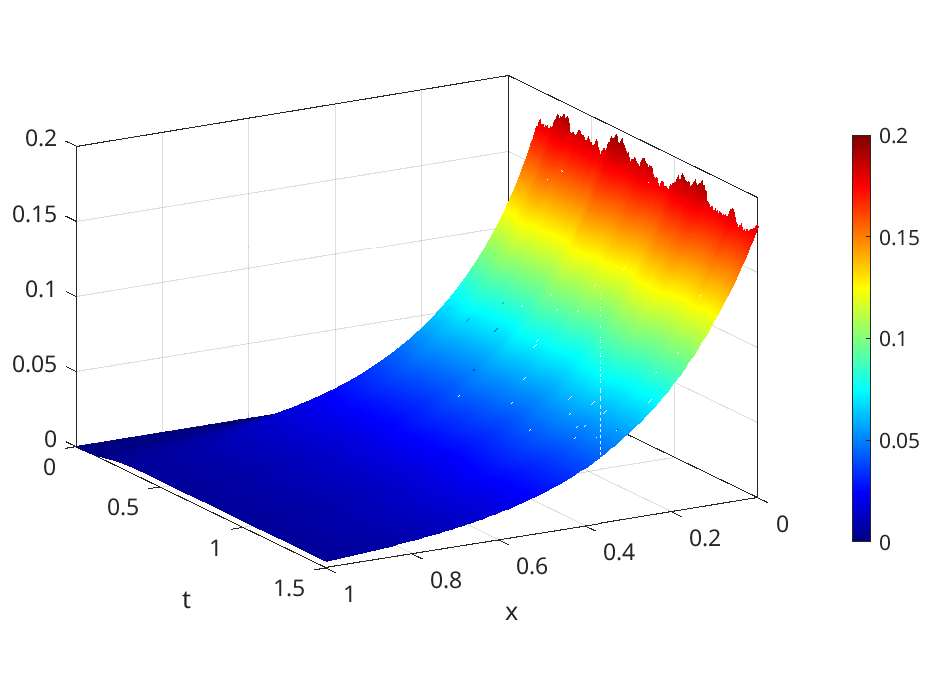}
    \end{subfigure}
    \begin{subfigure}{0.40\linewidth}
        \centering
        \includegraphics[width=\linewidth]{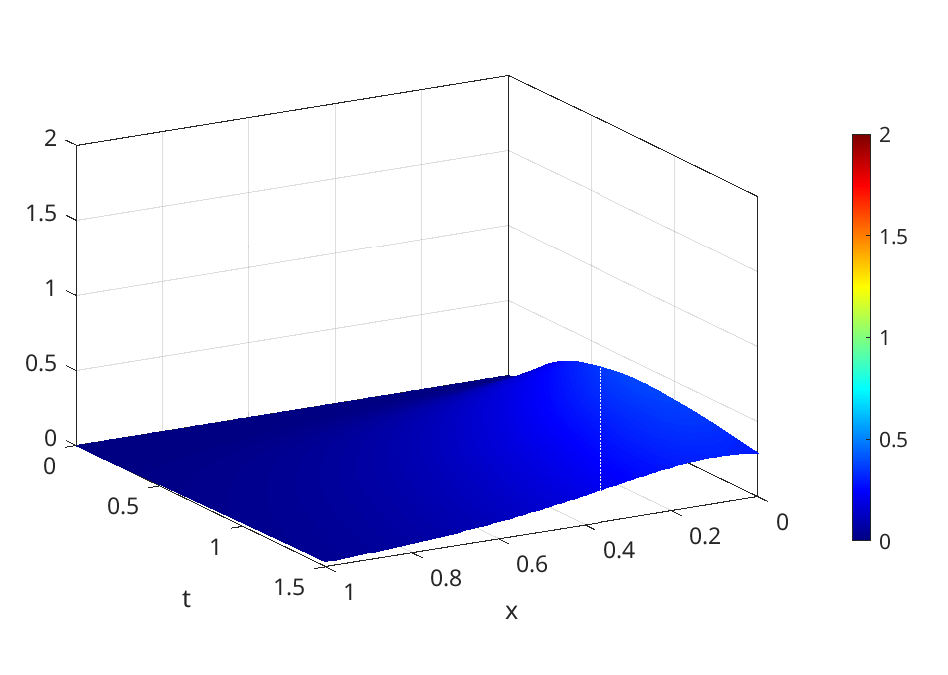}
    \end{subfigure}

    (a)

     \begin{subfigure}{0.40\linewidth}
        \centering
        \includegraphics[width=\linewidth]{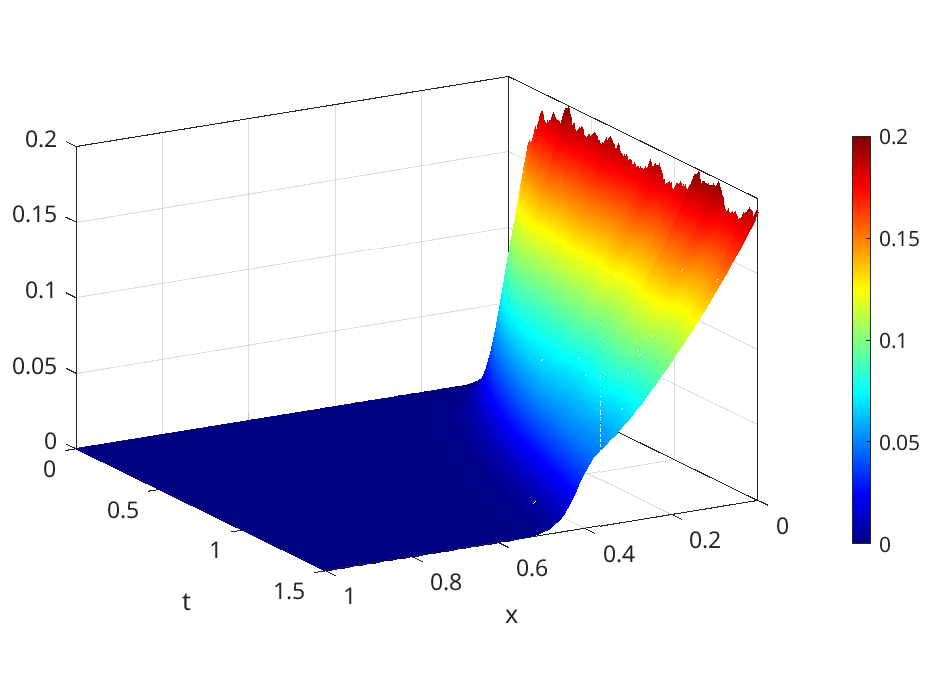}
    \end{subfigure}
    \begin{subfigure}{0.40\linewidth}
        \centering
        \includegraphics[width=\linewidth]{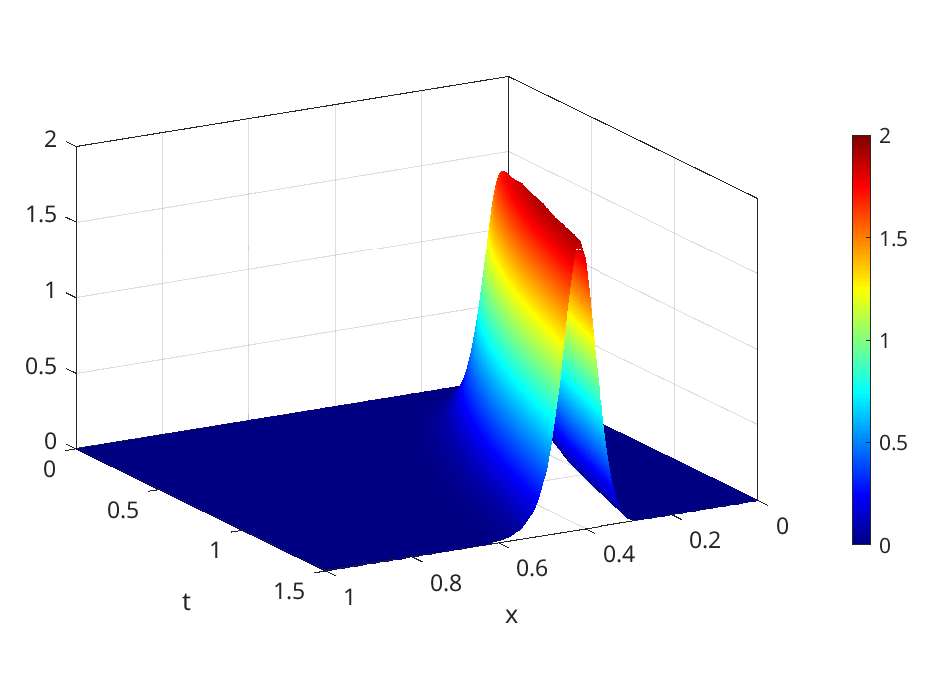} 
    \end{subfigure} 
      
    (b)

\caption{ Estimation of the distribution: standard deviation  $\overline{N}=500$  sample of   the processes $\rho$ (left) and $c$ (right)  at each mesh point. Parameters: $\varphi_1=0.2$, $\sigma_3=1$. (a) slow reaction $\lambda_1=1$; (b) fast reaction $\lambda_3=100.$  }	\label{fig:plot_standard_deviation_rho_c_different_lambda}
 \end{figure}
This is coherent with Figure \ref{fig:plot_standard_deviation_rho_c_different_lambda} which shows the standard deviations  $S_\rho, S_c$   of the empirical distribution at each point of the mesh    
\begin{equation*}%\label{eq:mean_var_rho_c}
\begin{split}
\left[S_\rho\right]^n_m = \sqrt{\frac{1}{\overline{N}-1}\sum_{j=1}^{\overline{N}} \left( \rho^n_m(\omega_j)-\overline{\rho}^n_m\right)^2}, &\qquad \overline{\rho}^n_m= \frac{1}{\overline{N}}\sum_{j=1}^{\overline{N}} \rho^n_m(\omega_j); \\\
 \left[S_c\right]^n_m= \sqrt{\frac{1}{\overline{N}-1}\sum_{j=1}^{\overline{N}} \left( c^n_m(\omega_j)-\overline{c}^n_m\right)^2}, &\qquad \overline{c}^n_m= \frac{1}{\overline{N}}\sum_{j=1}^{\overline{N}} s^n_m(\omega_j),  
 \end{split}
 \end{equation*}
for any $(m,n)\in \{0,\dots,M\}\times\{0,\ldots,N\}$ and the occurrences $\omega_j\in \Omega_p$, $j=1,...,\overline{N}.$
We can notice that the standard deviations of the synthetic samples for the sulphur dioxide concentration are of the same order close to the boundary where the influence of the noise is strong, both in the slow and fast regime. 
However, in the case of fast reaction, it vanishes soon. Concerning the calcium carbonate, the orders of magnitude are instead very different, precisely one order smaller in the case of slow reaction. Furthermore, when $\lambda_3=100$ the variance is always zero, except for an interval around the forming front. 
It seems that the spatial distribution of the variance is always the same but with a time-dependent translation. This is interesting and it is worth performing further experiments to identify a possible invariant distribution around the moving front.

 \begin{figure}[h!]
    \centering  
    \begin{subfigure}{0.37\linewidth}
        \centering
        \includegraphics[width=\linewidth]{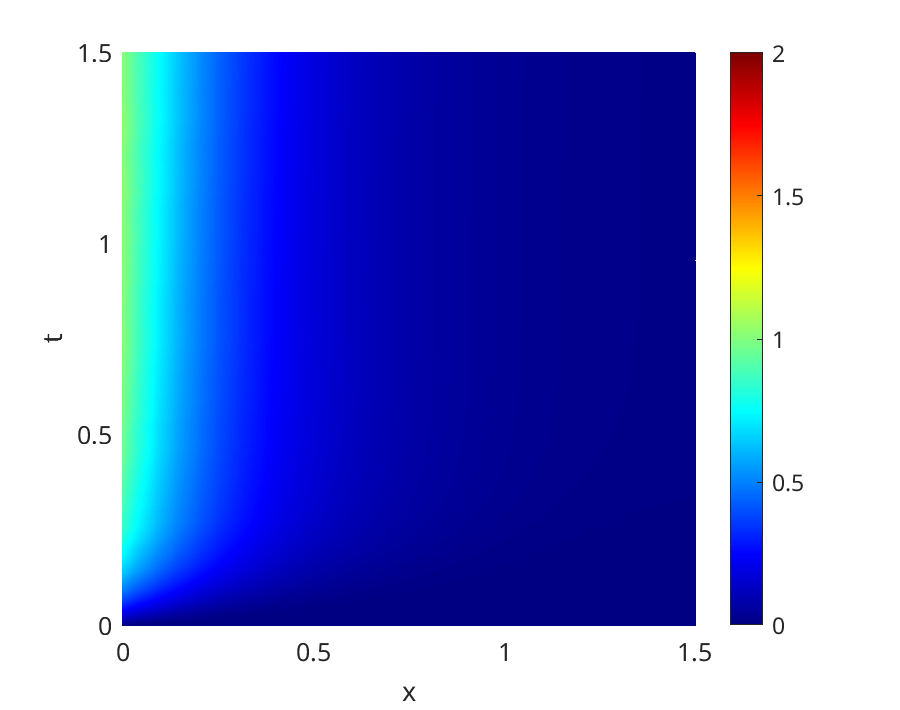}
    \end{subfigure}
    \begin{subfigure}{0.35\linewidth}
        \centering
        \includegraphics[width=\linewidth]{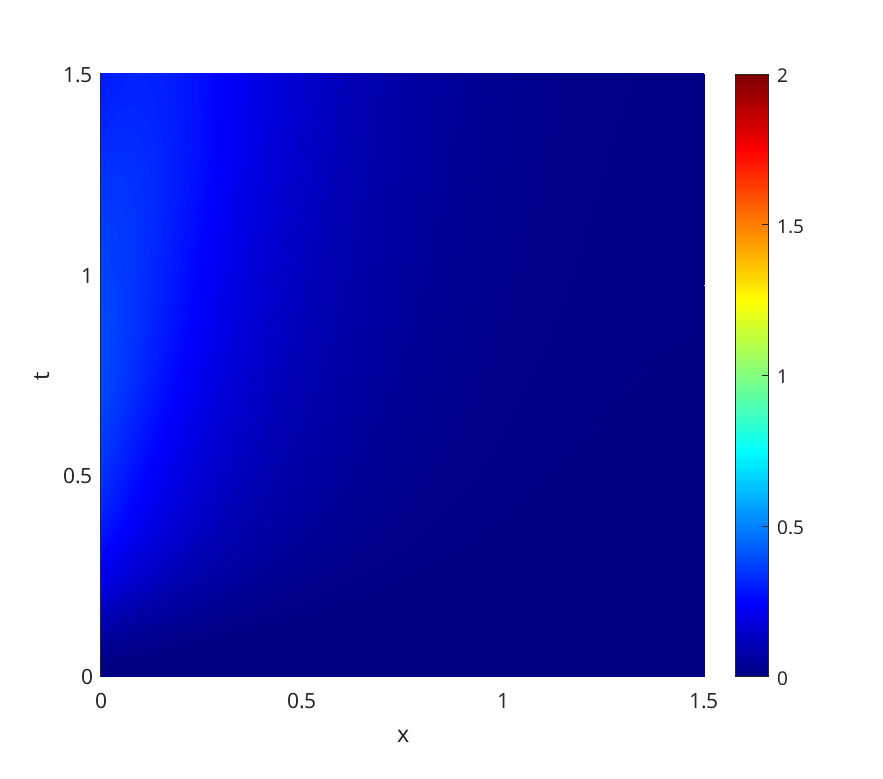}
    \end{subfigure}

    (a)

     \begin{subfigure}{0.35\linewidth}
        \centering
        \includegraphics[width=\linewidth]{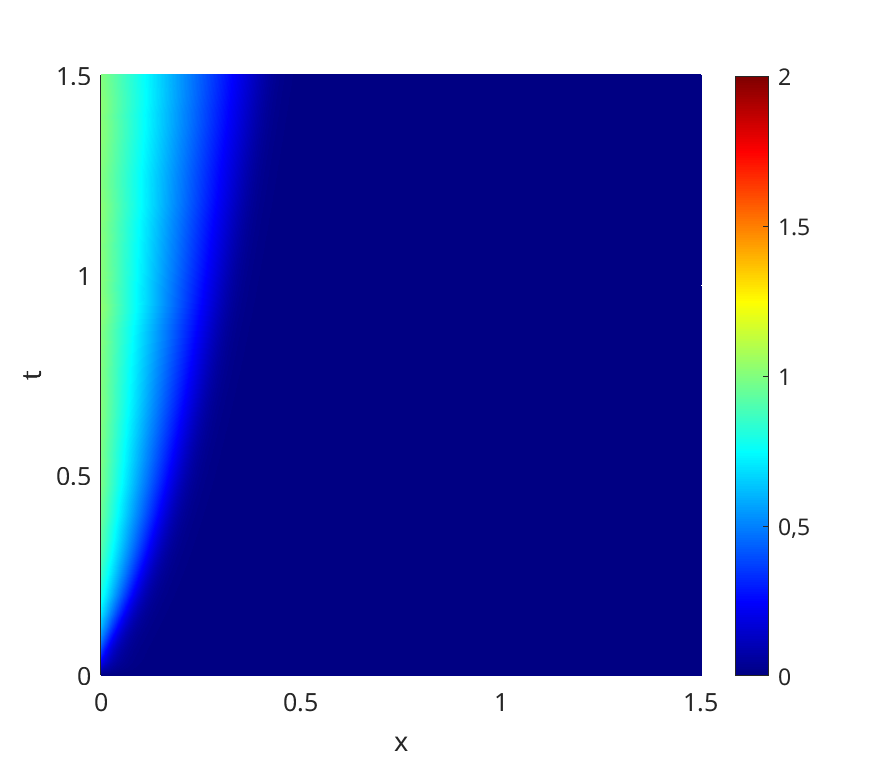}
    \end{subfigure}
    \begin{subfigure}{0.35\linewidth}
        \centering
        \includegraphics[width=\linewidth]{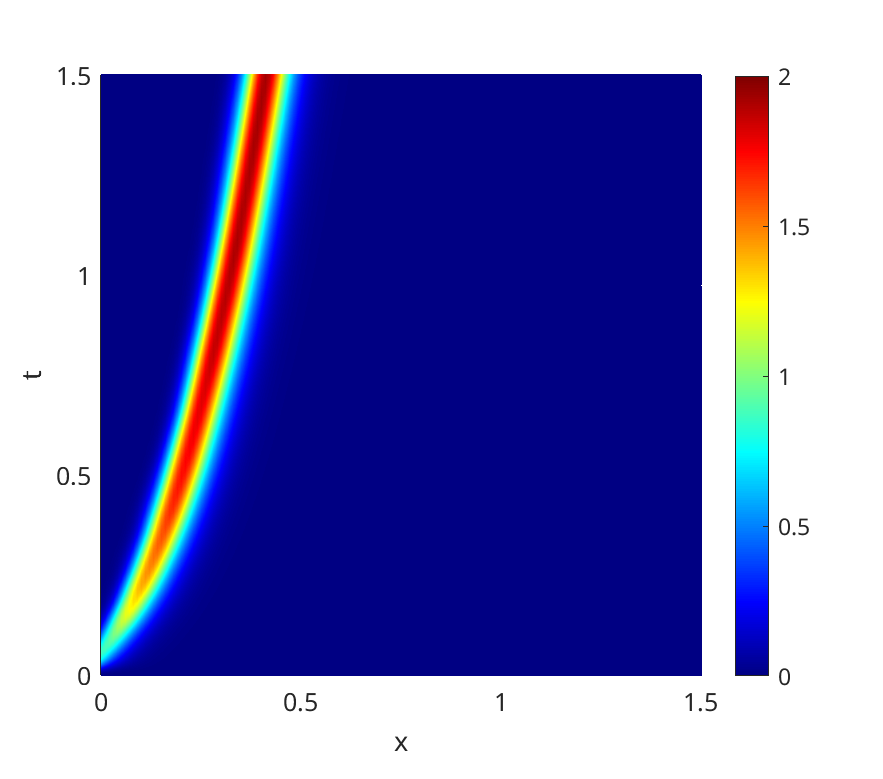} 
    \end{subfigure} 
      
    (b)

\caption{ Estimation of the impact of the noise at the boundary:  $RMSD(\rho_{\sigma},\rho_{\sigma_0})$ (left) and   $RMSD(c_{\sigma},c_{\sigma_0})$ (right)  at each mesh point. Parameters: $\varphi_1=0.2$, $\sigma_3=1$. (a) slow reaction $\lambda_1=1$; (b) fast reaction $\lambda_3=100.$  }	\label{fig:plot_colormap}
 \end{figure}
 
\medskip

The last measurement we perform aims to quantify the impact of the noise on the solution in terms of the computation of an $L^2(\Omega)$ distance between the deterministic solution $(\rho_{\sigma_0},c_{\sigma_0})$ and the noise solution $(\rho_{\sigma},c_{\sigma})$ for $\sigma_3=1$, both in the case of slow and fast reaction. We compute the \textit{root mean square difference} (RMSD),\cite{2016_burger} 
\begin{equation*}
\begin{split}
    RMSD(\rho_{\sigma},\rho_{\sigma_0})(t,x)&=\sqrt{\mathbb E\left[\left| \rho_{\sigma_3}(t,x)-\rho_{\sigma_0}(t,x)\right|^2\right]},\\
    RMSD(c_{\sigma},c_{\sigma_0})(t,x)&=\sqrt{E\left[\left| c_{\sigma_3}(t,x)-c_{\sigma_0}(t,x)\right|^2\right]}.
\end{split}
\end{equation*}

Figure \ref{fig:plot_colormap} shows the impact of the noise very clearly, displaying that the noise at the boundary propagates immediately throughout the entire domain: more slowly for the slow reaction rate. Moreover, in the latter configuration, the impact of the noise at the boundary for the calcium carbonate density is almost zero everywhere. On the other hand, in the case of fast reaction, there is a more significant impact of the noise for the calcite around the formation of the front, whereas it is almost zero for the \ch{SO2} concentration, coherently with the analysis of the variability previously discussed.

 \section{Conclusions}
In this work we have studied from the numerical point of view a nonlinear system on the half-line, recently proposed in   \cite{2023_Arceci_Giordano_Maurelli_Morale_Ugolini} and \cite{2023_MMU}, given by a parabolic reaction-diffusion equation and a reaction ODE, describing the phenomenon of marble sulphation.  The system is coupled with a stochastic dynamical boundary condition, given by a specific Pearson diffusion, which is original in the context of the degradation of cultural heritage. For completeness, we have discussed positiveness, boundedness, and stability for the Pearson stochastic differential equation,  and its numerical scheme, known as Lamperti Sloping Smooth Truncation, and based on the Euler-Maruyama method. The original part of the paper is the  numerical approximation of the overall system. According with the splitting strategy adopted in the study of the well-posedness of the system, provided in   \cite{2023_MMU}, we  have proposed to split the system into two subsystems: an autonomous heat equation which inherits the stochastic boundary condition, and a nonlinear system, coupled with the first, and therefore random, but with deterministic boundary conditions.  We present an FTCS approximation and discuss the stability conditions. Then the spatial accuracy order one is numerically estimated. An extensive qualitative investigation of the numerical solutions for the entire system is provided for different levels of noise at the boundary, different values of the porosity of the material, both in the slow and fast regime of the reaction. We discuss not only pathwise behaviour, but also distribution, first and second moments of the samples, showing how the variability at the boundary spreads out into the domain. One can see that as $\sigma$ increases, the noise impact on the entire domain also increases, gradually for the slower regime and  with high fluctuations for a faster reaction. In the latter case,  the transition zone gets smaller: the interior
remains unaffected by sulphur dioxide without penetration of the sulphur dioxide, and more  evident calcite deterioration at the boundary of the sample is 	highlighted with a subsequent formation of a moving front.

\section*{Acknowledgments}
The research is carried out within the research project PON 2021(DM 1061, DM 1062) ``Deterministic and stochastic mathematical modelling and data analysis within the study for the indoor and outdoor impact of the climate and environmental changes for the degradation of the Cultural Heritage" of the Università degli Studi di Milano. 
The authors are members of GNAMPA (Gruppo Nazionale per l’Analisi Matematica, la Probabilità e le loro Applicazioni) of the Italian Istituto Nazionale di Alta Matematica (INdAM).

\end{document}